\documentclass[sn-mathphys-num]{sn-jnl}

\usepackage{graphicx}%
\usepackage{multirow}%
\usepackage{amsmath,amssymb,amsfonts}%
\usepackage{amsthm}%
\usepackage{mathrsfs}%
\usepackage[title]{appendix}%
\usepackage{xcolor}%
\usepackage{textcomp}%
\usepackage{manyfoot}%
\usepackage{booktabs}%
\usepackage{algorithm}%
\usepackage{algorithmicx}%
\usepackage{algpseudocode}%
\usepackage{listings}%
\usepackage{arydshln}

\theoremstyle{thmstyleone}%
\newtheorem{theorem}{Theorem}[section]
\newtheorem{proposition}[theorem]{Proposition}%

\newtheorem{lemma}{Lemma}[section]

\theoremstyle{thmstyletwo}%

\theoremstyle{thmstylethree}%
\newtheorem{definition}{Definition}%

\begin{document}

\title[Article Title ]{Proximal Iterative Hard Thresholding Algorithm for Sparse Group $\ell_0$-Regularized Optimization with Box Constraints}

\author[1]{\fnm{} \sur{Yuge Ye}}\email{fujianyyg@163.com}

\author*[1]{\fnm{} \sur{Qingna Li}}\email{qnl@bit.edu.cn}

\affil[1]{\orgdiv{Department of Mathematics and Statistics}, \orgname{Beijing Institute of Technology}, \orgaddress{\street{No.5 Yard, Zhong Guan Cun South Street}, \city{Beijing}, \postcode{100081}, \state{Beijing}, \country{China}}}

\abstract{
This paper investigates a general class of problems in which a lower bounded smooth convex function incorporating $\ell_{0}$ and  $\ell_{2,0}$ regularization terms is minimized over box {constraints}.  Although such problems arise frequently in practical applications, their non-convexity introduced by $\ell_{0}$ and  $\ell_{2,0}$ regularization terms poses significant challenges for solution methods. In particular, we focus on the proximal operator associated with these regularization terms, which incorporates both group-sparsity and element-wise sparsity terms. Besides, we introduce the concepts of $\tau$-stationary point and support optimal (SO) point {and} analyze their relationship with the minimizer of the considered problem. Based on the proximal operator, we propose a novel proximal iterative hard thresholding algorithm to solve the problem. Furthermore, we establish the global convergence of our proposed method. Finally, extensive {numerical results} demonstrate the efficiency of our method.
}


\maketitle

\section{Introduction}\label{sec1}
\numberwithin{equation}{section}
Over the last few years, sparse optimization has {gained} considerable attention and has
been rigorously investigated by mathematicians {in the world}, primarily due to the emergence
of compressed sensing \cite{1542412}. Subsequently, an increasing number of applications of sparse optimization have been discovered {including} wireless communication \cite{10636212}, signal and image processing \cite{5447114}, machine learning \cite{pmlr-v32-qin14}, and artificial intelligence \cite{Deng2013}. Consequently, sparse optimization has evolved into an increasingly compelling and relevant field of inquiry.

In this paper, we consider the sparse solutions of a non-overlapping sparse group $\ell_0$-regularized optimization problem with box constraints. Suppose vector $x\in \mathbb{R}^n$ has a predefined non-overlapping group division. That is, $x = \left( x^{\top}_{G_1}, \cdots, x^{\top}_{G_q} \right)^{\top}$, where $x_{G_i}$ is the sub-vector of {$x$} restricted to $G_i$, $\cup^{q}_{i=1} G_i = \{1,\cdots,n\}$ and $G_i \cap G_j = \emptyset$, $\forall i\neq j$ and $q\in \{1,\cdots,n\}$. The optimization problem {we are interested in} is as follows:
\begin{align}\label{GBL0}
\min \ &  f(x) + \lambda \|x\|_0 + \mu \|x\|_{2,0},\
\mbox{s.t.} \  x\in \Omega,
\end{align}
where
\begin{align}\label{eq-arange}
\Omega := \{ x\in \mathbb{R}^n\, \mid \, -l\le x\le u,\ \ l,\ u \in \overline{\mathbb{R}}_{+}^{n}\},\ \overline{\mathbb{R}}_{+}^{n}:= \{ x\in \mathbb{R}^n \, \mid \, x\ge 0\}.
\end{align}
$f: \mathbb{R}^n \to \mathbb{R}$ is {a} smooth function and bounded from below on $\Omega$. $\lambda$, $\mu  \in \mathbb{R}_{+}$ {are} penalty parameters. $\|x\|_0$ is the $\ell_0$ norm of $x$, counting the number of non-zero elements of $x$. {$\|x\|_{2,0}:= \left\|(\|x_{G_1}\|_2,\cdots, \|x_{G_q}\|_2)\right\|_{0}$} is the number of nonzero groups in terms of $\ell_2$ norm. Group-sparsity is an important class of structured sparsity and is referred to as block-sparsity in compressed sensing \cite{5954192}.

\subsection{Sparse Group Lasso}
One relaxation of problem \eqref{GBL0} with $\Omega= \mathbb{R}^n$ is the following sparse group Lasso (SGLasso):
\begin{align}\label{GL1}
\min\limits_{x\in \mathbb{R}^n} & f(x) + \lambda \|x\|_1 + \mu \|x\|_{2,1}.
\end{align}
If $\mu = 0$, then \eqref{GL1} {becomes} the classical Lasso that was first proposed by Tibshirani \cite{10.1111/j.2517-6161.1996.tb02080.x}. Assuming prior knowledge of the group structure in the data, Yuan and Lin \cite{10.1111/j.1467-9868.2005.00532.x} proposed the group Lasso, that is, $\lambda=0$ in \eqref{GL1}. The group Lasso can guarantee the sparsity at the group level. However, it does not ensure the sparsity within each group. In practical problems, data often exhibit both overall element-wise sparsity and group sparsity. That is the reason why the SGLasso has been widely applied to different fields. For example, Simon et al. \cite{Simon01012013} validated the effectiveness of the mixed sparse LASSO  by comparing the predictive accuracy of LASSO, group LASSO, and SGLasso in gene expression studies. Vincent and Hansen \cite{VINCENT2014771} applied the SGLasso to classification problems based on multivariate regression, finding that the SGLasso classifier provides more sparse
solutions than the group LASSO. Furthermore, \eqref{GL1} has also been applied in various fields, including climate change prediction \cite{6137353}, feature selection with uncertain data \cite{xie2014sparse}, and target tracking \cite{7890458}. Currently, the algorithms for solving  \eqref{GL1} include coordinate descent algorithms \cite{friedman2010notegrouplassosparse, Laria03072019}, separable sparse reconstruction methods \cite{5464845}, semi-smooth Newton augmented Lagrange methods \cite{zhang2020efficient}, linearized alternating direction method of multipliers \cite{LI2014203} and smoothing composite proximal gradient algorithm for nonsmooth loss functions \cite{shen2024smoothing}.

\subsection{Sparse Group $\ell_0$-Regularized Optimization}
Although the $\ell_1$ norm is the best convex approximation of the $\ell_0$ norm, the $\ell_1$ norm often leads to excessive relaxation and biased estimation. Therefore, it is natural to consider directly solving the following sparse group $\ell_0$-regularized optimization
\begin{align}\label{GL0}
\min\limits_{x\in \mathbb{R}^n} & f(x) + \lambda \|x\|_0 + \mu \|x\|_{2,0}.
\end{align}
In the case where $\mu = 0$ in  \eqref{GL0}, it {reduces} to a common $\ell_0$-regularized optimization problem. For this case, numerous efficient algorithms have been proposed
\cite{zhou2021newton,cheng2020active,blumensath2008iterative}.
{For $\mu \neq 0$ and $\lambda \neq 0$, \eqref{GL0} also arises from various applications. One typical application comes from the multi-tissue decomposition for diffusion magnetic resonance imaging (MRI) signals \cite{7505644, 10.1007/978-3-319-24574-416}, where Yap et al. studied the $\ell_0$ sparse group estimation model in the form of \eqref{GL0} with a constraint ($x\ge 0$). Numerical results in \cite{7505644} verified that jointly using $\ell_{2,0}$ and $\ell_0$ regularization terms leads to better result than using only $\ell_0$ regularization (See \cite[Part B, page 4349]{7505644}.). Another application is the differential optical absorption spectroscopy (DOAS) analysis, which is formulated as \eqref{GL0} by Hu et al. \cite{hu2025iterative}, who proposed iterative mixed thresholding algorithm with a continuation technique (IMTC) to solve \eqref{GL0}. Numerical results demonstrated the superior performance of \eqref{GL0}, compared with only considering $\ell_0$ or $\ell_{2,0}$ regularization \cite[page 530 and Figure 2]{hu2025iterative}. Other applications include climate prediction \cite{doi:10.1137/1.9781611972825.5}, genetic association detecting \cite{10.1111/biom.12292}, neural network compression \cite{s22114081} and so on.} Liao et al. \cite{liao2024subspace} proposed a (SNSG) Newton's method for \eqref{GL0} and proved its global converge as well as local quadratic convergence rate.

The main difference between \eqref{GBL0} and \eqref{GL0} lies in the fact that \eqref{GBL0} includes extra box constraints $\Omega \subseteq \mathbb{R}^n$. {As far as we know, the box constraints} have been considered in \cite{doi:10.1137/18M1186009, 6307860} and have been shown to be beneficial for the recovery of images compared to approaches without such constraints \cite{6307860}. For \eqref{GBL0}, Li et al. \cite{doi:10.1137/21M1443455} developed a DCGL (Difference of-Convex algorithm for Sparse Group $\ell_0$ problem) based on the capped-$\ell_1$ function for solving it. For $\lambda = 0$ in \eqref{GBL0} with a convex constraints, Zhang and Peng \cite{zhang2022solving} proposed a (GSPG) group
smoothing proximal gradient algorithm for solving it. For $\mu = 0$ in \eqref{GBL0}, Lu \cite{lu2014iterative} provided a closed-form for the $\ell_0$-regularized problem and proposed a (PIHT) proximal iterative hard thresholding method. Based on PIHT, Wu and Bian \cite{wu2020accelerated} proposed an accelerated method for PIHT. Inspired by ITMC and PHIT, a natural question arises: Can we develop a direct method to solve \eqref{GBL0}? This {motivates} the work in this paper. We develop a Proximal Iterative Hard Thresholding method for the Sparse Group $\ell_0$-regularized problem with Box constraints (PIHT-SGB).

The main contributions of this paper are summarized as {follows}. Firstly, we provide a closed-form of the proximal operator for \eqref{GBL0}. Secondly, we introduce the concept of $\tau$-stationary point and support optimal (SO) point for \eqref{GBL0}. We explore a relationship among SO point, $\tau$-stationary point and  the minimizer of \eqref{GBL0}. This relationship plays an important role in establishing convergence analysis. Thirdly, we proposed the so-called PIHT-SGB for \eqref{GBL0}. Fourthly, we show that the sequence generated by our
method converges to a local minimizer of \eqref{GBL0}. Moreover, we establish the iteration complexity
of the PIHT-SGB method for finding a local-optimal solution. Finally, we demonstrate the efficiency of our proposed method through extensive numerical {results}.

The organization of the paper is as follows. In {Section} 2, we introduce the proximal operator for \eqref{GBL0} and propose the so-called PIHT-SGB for \eqref{GBL0}. In {Section} 3, we introduce the concept of $\tau$-stationary point and support optimal (SO) point of \eqref{GBL0}, then we do convergence and complexity analysis of our method. We conduct various numerical experiments in
{Section} 4 to verify the efficiency of the proposed method. Final conclusions are given in {Section} 5.

Notations. Let $\| \cdot \|$ represents the $\ell_2$ norm. For $x\in \mathbb{R}^n$, $|x| := ( |x_1|, |x_2|, \cdots, |x_n|)^{\top}$ denotes the absolute value of each component of $x $. Denote {$\mathcal{S}(x) = {\rm{supp}} (x)$} to be its support set consisting of the indices of the non-zero elements. For any $q\in \mathbb{Z}_{+}$, denote $[q] := [1,\cdots, q]$. Given a set $\mathcal{I} \subseteq [n]$, we denote $|\mathcal{I}|$ as its cardinality set and $\overline{\mathcal{I}}$ as its complementary set. Denote $\Pi_{\Omega}(x):= \arg \min\limits_{y\in \Omega}\|x-y\|^2 $ as a project mapping on {set} $\Omega$.

\section{Proximal Operator and Solution Method}
In this section, we provide some basic concepts and related properties for \eqref{GBL0}, and propose the so-called PIHT-SGB method.

\subsection{Proximal operator}
First, we transform \eqref{GBL0} into the following equivalent form
\begin{align}\label{PBL0}
\min\limits_{x\in \mathbb{R}^{n}} \phi(x) := f(x) + \lambda \|x\|_0 + \mu \|x\|_{2,0} + \delta_{\Omega}(x),
\end{align}
where $\Omega$ is defined by \eqref{eq-arange} {and} $\delta_{\Omega}(\cdot)$ denotes the indicator function of $\Omega$, {defined by} \begin{align*}
\delta_{\Omega}(x) =
\begin{cases}
0, \        &\mbox{if}\  x\in \Omega , \\
+\infty, \  &\mbox{if}\  x\notin \Omega .\\
\end{cases}
\end{align*}
Throughout this paper, we define {$(\tau >0)$}
\begin{align*}
p(\cdot) &:= \lambda \|\cdot \|_0 + \mu \|\cdot\|_{2,0} + \delta_{\Omega}(\cdot),\\
s_{\tau}(x) &:= x - \tau \nabla f(x),\\
d_{\tau}(x) &:= \Pi_{\Omega}(s_{\tau}(x)) - s_{\tau}(x) .
\end{align*}
We now review PIHT {(proximal iterative hard thresholding method) proposed in} \cite{lu2014iterative}, which is designed for the following problem
{\begin{align}\label{eq-BL0}
\min\limits_{x\in \mathbb{R}^n} \ &  f(x) + \lambda \|x\|_0,\
\mbox{s.t.} \  x\in \Omega.
\end{align}}
The framework of PIHT is as follows.
\begin{algorithm}[H]
\caption{PIHT for \eqref{eq-BL0}.}\label{alg-PIHT}
{Choose an arbitrary $x_0 \in \Omega$. Set $k=0$}.

(1) Solve the subproblem
\begin{align}\label{eq-subproblem-l0}
x^{k+1} \in \arg \min\limits_{{y\in \Omega}}\ \left\{f(x^k) + \langle \nabla f(x^k),y-x^k\rangle + \frac{1}{2\tau}\|y-x^k\|^2 +  \lambda\|y\|_0 \right\}.
\end{align}

(2) Set $k \gets k + 1$ and go to step (1).
\end{algorithm}
The subproblem \eqref{eq-subproblem-l0} has a closed form solution {as shown below}.
\begin{lemma}{\rm{\cite[Lemma 3.2]{lu2014iterative}}}\label{lem-proximall0}
For $i=1,\cdots, n$, denote $s_{\tau}(x^k) = x^k - \tau \nabla f(x^k)$. The solution $x^{k+1}$ of \eqref{eq-subproblem-l0} is given as follows:
\begin{align}\label{eq-proximall0-1}
x^{k+1}_i =
\begin{cases}
\left( \Pi_{\Omega}(s_{\tau}(x^k)) \right)_i,\ &\mbox{if}\ \left(s_{\tau}(x^k)\right)_i^2 - \left( d_{\tau}(x^k)\right)_i^2 > 2\lambda \tau, \\
0,\ &\mbox{if}\ \left(s_{\tau}(x^k)\right)_i^2 - \left( d_{\tau}(x^k)\right)_i^2 < 2\lambda \tau, \\
\left( \Pi_{\Omega}(s_{\tau}(x^k)) \right)_i\ \mbox{or}\ 0,\ &\mbox{if}\ \left(s_{\tau}(x^k)\right)_i^2 - \left( d_{\tau}(x^k)\right)_i^2 = 2\lambda \tau.
\end{cases}
\end{align}
\end{lemma}
As we know, for a proper lower semicontinuous function $h: \mathbb{R}^n \to (-\infty,\,  +\infty]$, its proximal operator ${\rm{Prox}}_{h(\cdot)}: \mathbb{R}^n \to \mathbb{R}^n$ is a set-valued mapping defined by
\begin{align}\label{eq-Def-proximal}
{\rm{Prox}}_{h(\cdot)}(s) :=\arg \min\limits_{y\in \mathbb{R}^n} \frac{1}{2}\|y - s\|^2 + h(y).
\end{align}
From Lemma \ref{lem-proximall0}, we can easily obtain the following {result}.
\begin{proposition}\label{pro-proximall0}
The proximal operator ${\rm{Prox}}_{\tau \lambda {\|\cdot\|_0}}(\cdot)$ admits a closed form as
\begin{align}\label{eq-proximall0-2}
\left({\rm{Prox}}_{\tau \lambda\|\cdot\|_0}(s) \right)_{i} =
\begin{cases}
\left( \Pi_{\Omega}(s) \right)_i,\ &\mbox{if}\ \left(s\right)_i^2 - \left( \Pi_{\Omega}(s) - s\right)_i^2 > 2\lambda \tau, \\
0,\ &\mbox{if}\ \left(s\right)_i^2 - \left( \Pi_{\Omega}(s) - s\right)_i^2 < 2\lambda \tau, \\
\left( \Pi_{\Omega}(s) \right)_i\ \mbox{or}\ 0,\ &\mbox{if}\ \left(s\right)_i^2 - \left( \Pi_{\Omega}(s) - s\right)_i^2 = 2\lambda \tau.
\end{cases}
\end{align}
\end{proposition}
To give the closed form for ${\rm{Prox}}_{\tau p(\cdot)}(\cdot)$, we need the following {result}.
\begin{proposition} \rm{\cite[Lemma 2.1]{calamai1987projected}} \label{pro-projection}
{Let} $\Omega \subset \mathbb{R}^n$ {be} a nonempty closed convex set. Let  $\Pi_{\Omega}(\cdot)$ be the projection into $\Omega$.
\begin{description}
\item{(i)} If $z \in \Omega$, then
$\left\langle \Pi_{\Omega}(x) - x,\,  z - \Pi_{\Omega}(x) \right\rangle \ge 0,\ \forall x \in \mathbb{R}^n .$

\item{(ii)} $\Pi_{\Omega}(\cdot)$ is a monotone operator, that is, $\langle \Pi_{\Omega}(y) - \Pi_{\Omega}(x),\, y - x\rangle \ge 0$ for any $x$,  $y\in \mathbb{R}^n$. If $\Pi_{\Omega}(y) \neq \Pi_{\Omega}(x)$, then {the} strict inequality holds.

\item{(iii)} $\Pi_{\Omega}(\cdot)$ is a nonexpansive operator, that is, $\|\Pi_{\Omega}(y) - \Pi_{\Omega}(x)\| \le \|y-x\|$ for any $x$, $y\in \mathbb{R}^n$.
\end{description}
\end{proposition}
The following theorem provides a closed form for ${\rm{Prox}}_{\tau p(\cdot)}(\cdot)$.
\begin{theorem}\label{the-proximalgroup}
The operator ${\rm{Prox}}_{\tau p(\cdot)}(\cdot)$ takes a closed form as {follows}
\begin{align}\label{eq-the-groupproximal}
\left({\rm{Prox}}_{\tau p(\cdot)}(s) \right)_{G_i} =
\begin{cases}
z_{G_i},\ &\mbox{if}\ \|z_{G_i}\| > \sqrt{2\tau (\lambda \|z_{G_i}\|_0 + \mu )}, \\
z_{G_i}\, \mbox{or}\  0,\ &\mbox{if}\ \|z_{G_i}\| = \sqrt{2\tau (\lambda \|z_{G_i}\|_0 + \mu )}, \\
0,\ &\mbox{otherwise},
\end{cases}
\end{align}
where $z \in {\rm{Prox}}_{\tau \lambda \|\cdot \|_0}(s)$.
\end{theorem}
\begin{proof}
We need to consider
\begin{align}\label{eq-pro-1}
\min\limits_{x \in \mathbb{R}^n} {\Psi}(x;\, s) := \frac{1}{2}\|x - s\|^2 + \tau \lambda \|x\|_{0} + \tau \mu \|x\|_{2,0} + \delta_{\Omega}(x) .
\end{align}
Note that \eqref{eq-pro-1} is of a group separable structure. The solution of \eqref{eq-pro-1} can be achieved parallelly at each group. Therefore, we only need to consider the following subproblems
\begin{align}\label{eq-pro-group}
\min\limits_{x\in \mathbb{R}^{n_i}} \Psi_i(x;\, s_{G_i}) := \frac{1}{2}\|x - s_{G_i}\|^2 + \tau \lambda \|x\|_{0} + \tau \mu \|x\|_{2,0} + \delta_{D_i}(x),
\end{align}
where $i\in [p]$, $D_i := \left\{ x\in \mathbb{R}^{|G_i|}\, \big| \, -l_{G_i} \le x \le u_{G_i} \right\}$.
For each $x\in D_i \setminus \{0\}$, it holds that $\|x\|_{2,0} = 1$. Thus, we obtain that
\begin{align}\label{eq-pro-l0}
\Psi_i(x;\, s_{G_i}) = \frac{1}{2}\|x - s_{G_i}\|^2 + \tau \lambda \|x\|_{0} + \tau \mu + \delta_{D_i}(x),\ {x\in D_i \setminus \{0\}}.
\end{align}
By Proposition \ref{pro-proximall0}, a minimum of $\Psi_i(\cdot; s_{G_i})$ over $D_i \setminus \{0\}$ is $z_{G_i} \in \left( {\rm{Prox}}_{\tau \lambda \|\cdot\|_0}(s)\right)_{G_i}$. By \eqref{eq-proximall0-2}, the non-zero component of $z_{G_i}$ is $\left(\Pi_{\Omega}(s)\right)_{G_i}$. For simplicity, let's assume {$z_{G_i}= \left(\Pi_{\Omega}(s)\right)_{G_i}$}, which does not affect the overall analysis. By Proposition \ref{pro-projection}, it holds that
$\left\langle \left(\Pi_{\Omega}(s)\right)_{G_i} - s_{G_i},\, 0 - \left(\Pi_{\Omega}(s)\right)_{G_i} \right\rangle \ge 0,$ hence $\left\langle  s_{G_i},\, z_{G_i} \right\rangle \ge \left\| \left(\Pi_{\Omega}(s)\right)_{G_i} \right\|^2
$. Together with \eqref{eq-pro-l0}, it holds that
\begin{align*}
\Psi_i(z_{G_i};\, s_{G_i}) &= \frac{1}{2}\|z_{G_i} - s_{G_i}\|^2 + \tau \lambda \|z_{G_i}\|_{0} + \tau \mu \\
&= \frac{1}{2}\|z_{G_i}\|^2 + \frac{1}{2}\|s_{G_i}\|^2 - \left\langle z_{G_i}, s_{G_i}\right\rangle + \tau \lambda \|z_{G_i}\|_{0} + \tau \mu \\
&\le  \frac{1}{2}\|z_{G_i}\|^2 +\frac{1}{2}\|s_{G_i}\|^2 - \|z_{G_i}\|^2 + \tau \lambda \|z_{G_i}\|_{0} + \tau \mu \\
&= \frac{1}{2}\|s_{G_i}\|^2 - \frac{1}{2} \|z_{G_i}\|^2 + \tau \lambda \|z_{G_i}\|_{0} + \tau \mu .
\end{align*}
For $x = 0$, it holds that $\Psi_i(0;\,  s_{G_i}) = \frac{1}{2} \|s_{G_i}\|^2$. It is obvious that
$$\Psi_i(z_{G_i};\,  s_{G_i}) - \Psi_i(0;\, s_{G_i}) = - \frac{1}{2} \|z_{G_i}\|^2 + \tau \lambda \|z_{G_i}\|_{0} + \tau \mu  . $$
Now, we consider three cases from \eqref{eq-the-groupproximal}.

\textbf{Case 1}. If $\|z_{G_i}\| > \sqrt{2 \tau (\lambda \|z_{G_i}\|_{0} + \mu) }$, then $ \Psi_i(z_{G_i};\, s_{G_i}) <  \Psi_i(0;\, s_{G_i})$. The minimizer of \eqref{eq-pro-group} is $z_{G_i}$.

\textbf{Case 2}. If $\|z_{G_i}\| < \sqrt{2 \tau (\lambda \|z_{G_i}\|_{0} + \mu) }$, then $ \Psi_i(z_{G_i}; s_{G_i}) >  \Psi_i(0; s_{G_i})$. The minimizer of \eqref{eq-pro-group} is 0.

\textbf{Case 3}. If $\|z_{G_i}\| = \sqrt{2 \tau (\lambda \|z_{G_i}\|_{0} + \mu) }$, then $ \Psi_i(z_{G_i}; s_{G_i}) =  \Psi_i(0; s_{G_i})$. Both $z_{G_i}$ and 0 are the minimizer of \eqref{eq-pro-group}. The proof is complete.
\end{proof}

\subsection{Solution method: PIHT-SGB}
Our objective is to design an algorithm that can converge to some stationary point of \eqref{PBL0}. {To that end}, similar to \cite{liao2024subspace}, we introduce a novel $\tau$-stationary point for \eqref{PBL0}.
\begin{definition}
{\rm{($\tau$-stationary point)}}. Let $\tau > 0$, a vector $x^* \in \Omega$ is called {a} $\tau$-stationary point of \eqref{PBL0} if the following holds
\begin{align}\label{eq-stationary}
x^* \in {\rm{Prox}}_{\tau p(\cdot)}(x^* - \tau \nabla f(x^*)).
\end{align}
\end{definition}
The main idea of our method is to find a point in each iteration which satisfies $$x^{k+1}\in {\rm{Prox}}_{\tau p(\cdot)}(s_{\tau}(x^k)).$$
To achieve this, we need to define
\begin{align}
\textbf{H} (y; \gamma) :=
\left\{ x\in \mathbb{R}^{n}\ \big| \, x_i = \left(\Pi_{\Omega}(y)\right)_i ,\ \text{if}\ |y_i|^2 - \left( \Pi_{\Omega}(y) - y \right)^2_i > \gamma ; \nonumber \right.\\
\left. x_i = 0,\ \text{if}\ |y_i|^2 - \left( \Pi_{\Omega}(y) - y \right)^2_i \le \gamma \right\}, \label{L0-hard}\\
\textbf{H}_{G_i}(z; \gamma) := \left\{x\in \mathbb{R}^{n_i}\ \big| \ x=z,\ \text{if}\ \|z\| > \gamma ;\
x=0,\ \text{if}\ \|z\| \le \gamma \right\} .\label{group-hard}
\end{align}
where $y\in \mathbb{R}^n$, $z\in \mathbb{R}^{n_i}$, $\gamma \ge 0$. The proposed algorithm is given as follows.
\begin{algorithm}[H]
\caption{PIHT-SGB for \eqref{GBL0}.}\label{alg-PIHT-SGB}
{Choose an arbitrary $x_0 \in \Omega$. Set $\lambda$, $\mu$, $\tau \in \mathbb{R}^n_{+}$. Set $k = 0$.}

(1) Solve the subproblem
\begin{align}
&s_{\tau}(x^k):= x^k - \tau \nabla f(x^k) ,\nonumber\\
&z^k:= \textbf{H}\left(s_{\tau}(x^k);\, 2\lambda \tau \right) , \label{group-iteration-2} \\
&x^{k+1}_{G_i}:= \textbf{H}_{G_i}\left( z^k_{G_i};\,  \sqrt{2\tau(\lambda \|z^k_{G_i}\|_0 + \mu)} \right) . \label{group-iteration-3}
\end{align}

(2) Set $k \gets k + 1$ and go to step (1).
\end{algorithm}
\begin{proposition}\label{pro-proximalgroup}
Suppose $\{x^k\}$ is generated by Algorithm \ref{alg-PIHT-SGB}. Then it holds that
\begin{align}\label{eq-groupproximal-equi}
x^{k+1} \in {\rm{Prox}}_{\tau p(\cdot)}(x^k - \tau \nabla f(x^k)).
\end{align}
Moreover, if $\|x^{k+1} - x^{k}\|=0$, then $x^{k}$ is a $\tau$-stationary point.
\end{proposition}
\begin{proof}
By \eqref{L0-hard} and Proposition \ref{pro-proximall0}, for $z^k:= \textbf{H}\left(s_{\tau}(x^k);\, \ {2\lambda \tau} \right)$, it is obvious that $z^k \in {\rm{Prox}}_{\tau \lambda \|\cdot\|_0}(x^k - \tau \nabla f(x^k))$. Consider $x^{k+1}_{G_i}:= \textbf{H}_{G_i}\left( z^k_{G_i};\, \  \sqrt{2\tau(\lambda \|z^k_{G_i}\|_0 + \mu)} \right)$, it is not hard to see that $x^{k+1}_{G_i}$ is a special case in \eqref{eq-the-groupproximal}. Together with {Theorem \ref{the-proximalgroup}}, we obtain \eqref{eq-groupproximal-equi}.

If ${\rm{Prox}}_{\tau p(\cdot)}$ is a singleton, then there is no index $i\in [q]$ such that $ \|z^k_{G_i}\| = \sqrt{2\tau (\lambda \|z^k_{G_i}\|_0 + \mu )}$ from \eqref{eq-the-groupproximal}. Similarly, there is no index $i\in [n]$ such that $\left(s_{\tau}(x^k)\right)_i^2 - \left( \Pi_{\Omega}(s_{\tau}(x^k)) - s_{\tau}(x^k)\right)_i^2 = 2\lambda \tau$ from \eqref{eq-proximall0-2}. Therefore the mapping $\textbf{H} (\cdot;\, {2\lambda \tau} )$ in \eqref{L0-hard} is equivalent to ${\rm{Prox}}_{\tau \lambda\|\cdot\|_0}(\cdot)$ in \eqref{eq-proximall0-2}. Similarly, $\textbf{H}_{G_i}\left( \cdot;\  \sqrt{2\tau(\lambda \|\cdot\|_0 + \mu)} \right)$ is equivalent to \eqref{eq-the-groupproximal} by \eqref{group-hard} and \eqref{eq-the-groupproximal}. {Combining with Theorem \ref{the-proximalgroup}, we obtain \eqref{eq-groupproximal-equi} }.

If $x^{k+1} = x^{k}$, together with \eqref{eq-groupproximal-equi}, {it holds that} $x^{k} \in {\rm{Prox}}_{\tau p(\cdot)}(x^k - \tau \nabla f(x^k))$. By the definition of $\tau$-stationary point in \eqref{eq-stationary}, the proof is complete.
\end{proof}
\section{Convergence and Complexity Analysis}
{In this section}, we {will} analyze the convergence and complexity of our method. We first provide the connection between \eqref{GBL0} and a constrained convex optimization problem regarding the global minimum point. Then we will prove that the number of changes in the support set of {the} iteration sequence is finite {and the sequence generated by the algorithm converges} to a local minimizer. Finally, we conduct an iterative complexity analysis of our method.

Our subsequent results require the strong smoothness and convexity of $f$.
\begin{definition}
{$f$ is strongly smooth with constant $L>0$} if
\begin{align}\label{eq-2-4}
f(z) \le f(x) + \langle \nabla f(x),z-x\rangle + \frac{L}{2}\|z-x\|^2,\ \forall \ x,\ z\in \mathbb{R}^n ,
\end{align}
$f$ is strongly convex {with} constant $\ell > 0$ if
\begin{align}\label{eq-2-5}
f(z) \ge f(x) + \langle \nabla f(x),z-x\rangle + \frac{\ell}{2}\|z-x\|^2,\ \forall \ x,\ z\in \mathbb{R}^n .
\end{align}
\end{definition}
Throughout this paper, we set {$\tau<\frac{1}{L}$}.
\subsection{Related constrained problem}
{To facilitate our analysis}, we introduce the definition of support optimal (SO) point from \cite{doi:10.1137/17M1116544}.
\begin{definition}
{\rm{(support optimality)}}. A vector $x^* \in \Omega$ is called a support optimal (SO) point of \eqref{PBL0} if $x^*$ is a global minimizer of {the following problem}
\begin{align}\label{eq-GSO}
\min\limits_{x\in \mathbb{X}} f(x),
\end{align}
where $\mathbb{X} := \{ x\in \Omega\, \mid \, {\rm{supp}} (x) \subseteq {\rm{supp}} (x^*) \}$.
\end{definition}
The following {lemma} demonstrates the relationship between a support optimal (SO) point and the minimizer of \eqref{PBL0}.
\begin{lemma}\label{lem-GBL0GSO}
Let $x^* \in \Omega$. Then, the following statements hold.
\begin{description}
\item{(i)} If $x^*$ is a global minimizer of \eqref{PBL0}, then $x^*$ is a SO point.
\item{(ii)} {If $f$ is convex and $x^*$ is a SO point} of \eqref{PBL0}, then $x^*$ is a local minimizer of \eqref{PBL0}.
\end{description}
\end{lemma}
\begin{proof}
(i). For any $x\in \mathbb{X}$, it holds that {$\|x\|_0 \le \|x^*\|_0$}, $\|x\|_{2,0} \le \|x^*\|_{2,0}$ and $\delta_{\Omega}(x) = \delta_{\Omega}(x^*) = 0$. Together with the {fact that $x^*$ is a global minimizer of \eqref{PBL0}}, we obtain that {for any $x\in \mathbb{X}$,}
\begin{align*}
f(x^*) + \lambda \|x^*\|_0 + \mu \|x^*\|_{2,0} \le f(x) + \lambda \|x\|_0 + \mu \|x\|_{2,0} \le f(x) + \lambda \|x^*\|_0 + \mu \|x^*\|_{2,0}.
\end{align*}
Thus, $f(x^*) \le f(x)$ holds over $\mathbb{X}$ which implies that $x^*$ is a global minimizer of  \eqref{eq-GSO}.

(ii). {Let} $x^*$ {be} a global {minimizer} of \eqref{eq-GSO}. Denote $\mathcal{S}_* := {\rm{supp}} (x^*)$. {For the case where $x^*\neq 0$, define
\begin{align*}
\epsilon := \min \left\{ \lambda / \|\nabla f(x^*)\|,\ \min\limits_{i\in \mathcal{S}_*} |x^*_i| \right\} {>} 0.
\end{align*}
Then it suffices to show that, for each $x\in \mathbb{B}_{\epsilon} (x^*) := \{ x\in \Omega \mid \, \|x-x^*\| {<} \epsilon \}$, the following holds
\begin{align}\label{eq-phix}
\phi(x) \ge \phi(x^*).
\end{align}
{To that end}, for any $x\in \mathbb{B}_{\epsilon} (x^*)$, it is obvious that $\delta_{\Omega}(x)=\delta_{\Omega}(x^*)=0$. Now, we claim $\mathcal{S}_* \subseteq {\rm{supp}} (x)$. { Indeed}, if there is $j\in \mathcal{S}_*$ {such that} $j\notin {\rm{supp}} (x)$, then we derive a contradiction
\begin{align*}
{\epsilon \le \min\limits_{i\in \mathcal{S}_*} |x^*_i| \le |x^*_j| = |x^*_j -x_j| \le \|x-x^*\| <} \epsilon .
\end{align*}
Therefore $\mathcal{S}_* \subseteq {\rm{supp}} (x)$. If $\mathcal{S}_* = {\rm{supp}} (x)$,  then $\|x\|_0 = \|x^*\|_0 $ and $\|x\|_{2,0} = \|x^*\|_{2,0}$. It follows from \eqref{eq-GSO} that $f(x) \ge f(x^*)$. It holds that
\begin{align*}
f(x) + \lambda \|x\|_0 + \mu \|x\|_{2,0} + \delta_{\Omega}(x) \ge f(x^*) + \lambda \|x^*\|_0 + \mu \|x^*\|_{2,0} + \delta_{\Omega}(x^*).
\end{align*}
That is, \eqref{eq-phix} holds. If $\mathcal{S}_* \subset {\rm{supp}} (x)$, then $\|x\|_0 \ge \|x^*\|_0 +1$ and $\|x\|_{2,0} \ge \|x^*\|_{2,0}$. Combining with the convexity of $f$, we obtain that
\begin{align}\label{eq-convexphi}
\phi(x) - \phi(x^*) &\ge f(x) - f(x^*) + \lambda \nonumber\\
&\ge \langle \nabla f(x^*),x-x^*\rangle + \lambda \nonumber\\
&\ge -\|\nabla f(x^*)\|\|x-x^*\| + \lambda .
\end{align}
Because $x\in \mathbb{B}_{\epsilon} (x^*)$, it holds that
\begin{align}\label{eq-convexphi2}
\phi(x) - \phi(x^*) \ge -\|\nabla f(x^*)\|\|x-x^*\| + \lambda \ge -\epsilon \|\nabla f(x^*)\| + \lambda \ge 0.
\end{align}
Hence, \eqref{eq-phix} is proved for each $x\in \mathbb{B}_{\epsilon} (x^*)$. }

{For the case where $x^* = 0$, let $\epsilon := \lambda / \|\nabla f(x^*)\|$. Then $\epsilon > 0$. It is obvious that $\mathcal{S}_* = \emptyset \subset {\rm{supp}}(x)$. Similarly, it holds that $\|x\|_0 \ge \|x^*\|_0 +1$ and $\|x\|_{2,0} \ge \|x^*\|_{2,0}$, which gives \eqref{eq-convexphi} and \eqref{eq-convexphi2} again. Therefore, \eqref{eq-phix} holds.
The proof is complete.}
\end{proof}
The sufficient decrease {lemma} for the proximal gradient mapping is given below.
\begin{lemma}{\rm{\cite[Lemma 2]{bolte2014proximal}}}\label{lem-Fdecreasing}
Let $f$ be strongly smooth with constant $L>0$. Let $\frac{1}{\tau} > L$ and $x \in \Omega$. For $ y\in
{\rm{Prox}}_{\tau p(\cdot)}(x - \tau \nabla f(x))$, it holds that
\begin{align}\label{eq-Fdecreasing0}
\phi(x) - \phi(y)\ge \frac{1/\tau - L}{2} \|y - x\|^2 .
\end{align}
\end{lemma}
The {following} result shows the connection among $\tau$-stationary point, SO point and the minimizer of \eqref{PBL0}.
\begin{lemma}\label{lem-stationarySO}
Let $x^*\in \Omega$. The following results hold.
\begin{description}
\item{(i)} If $x^*$ is {a} global minimizer of \eqref{PBL0}, then for any $\tau$ satisfing $\frac{1}{\tau} > L$, $x^*$ is a $\tau$-stationary point.
\item{(ii)} Suppose that $f$ is convex. Let $x^*$ be {a} $\tau$-stationary point of \eqref{PBL0} for some $\tau >0$. Then $x^*$ is a SO point of \eqref{PBL0}.
\end{description}
\end{lemma}
\begin{proof}
(i). Let $1/\tau > L$ and $y\in {\rm{Prox}}_{\tau p(\cdot)}(s_{\tau}(x^*))$. By Lemma \ref{lem-Fdecreasing} and the optimality of $x^*$, it holds that
\begin{align*}
f(x^*) + p(x^*) &\ge \frac{1/\tau - L}{2} \|y-x^*\|^2 + f(y) + p(y) \\
&\ge \frac{1/\tau - L}{2} \|y-x^*\|^2 + f(x^*) + p(x^*).
\end{align*}
Since $\frac{1}{\tau} > L$, we conclude that $y = x^*$. Consequently, $y$ is a $\tau$-stationary point.

(ii). Denote $\mathbb{X} := \{y\in \Omega \, \mid \, \mathcal{S}(y) \subseteq \mathcal{S}(x^*) \}$. By \eqref{eq-stationary} and \eqref{eq-Def-proximal}, it holds that
\begin{align*}
&\ \ \ \ \frac{1}{2}\|x^* - s_{\tau}(x^*)\|^2 + \tau \lambda\|x^*\|_{{0}} + \tau \mu\|x^*\|_{2,0} \\
&{=} \min\limits_{y\in \Omega} \frac{1}{2}\|y - s_{\tau}(x^*)\|^2 + \tau\lambda \|y\|_{{0}} + \tau \mu\|y\|_{2,0} \\
&\le \min\limits_{y\in \mathbb{X}} \frac{1}{2}\|y - s_{\tau}(x^*)\|^2 + \tau \lambda\|{y}\|_0 + \tau \mu\|{y}\|_{2,0}.
\end{align*}
It is obvious that, if $\mathcal{S}(y) \subseteq \mathcal{S}(x^*)$, {then} $\|y\|_0 \le \|x^*\|_{0}$, $\|y\|_{2,0} \le \|x^*\|_{2,0}$. Therefore, it holds that
\begin{align}\label{eq-staule}
&\ \ \ \ \frac{1}{2}\|y - s_{\tau}(x^*)\|^2 + \tau \lambda\|y\|_0 + \tau \mu\|y\|_{2,0} \nonumber\\
&\le \min\limits_{y\in \mathbb{X}} \frac{1}{2}\|y - s_{\tau}(x^*)\|^2 + \tau \lambda\|x^*\|_0 + \tau \mu\|x^*\|_{2,0}  ,
\end{align}
which gives
\begin{align*}
&\ \ \ \ \frac{1}{2}\|x^* - s_{\tau}(x^*)\|^2 \le \min\limits_{y\in \mathbb{X}} \frac{1}{2}\|y - s_{\tau}(x^*)\|^2 .
\end{align*}
It means that $x^* = \Pi_{\mathbb{X}} (s_{\tau}(x^*))$. From {\rm \cite[Proposition 2.3]{Pang1990}}, we know that $x^*$ satisfies the first order necessary condition for \eqref{eq-GSO}. Together with the convexity of $f$, $x^*$ is a global minimizer of \eqref{eq-GSO}. The proof is complete.
\end{proof}
To prove that the nonzero components of the sequence $\{x^k\}$ have a lower bound, we need the following {lemma}.
\begin{lemma}{\rm{\cite[Lemma 3.3]{lu2014iterative}}}\label{lem-notsmall10}
Let $\{z^k\}$ be generated by \eqref{group-iteration-2} in Algorithm \ref{alg-PIHT-SGB}. For all $k\ge 0$, the following holds
\begin{align}\label{eq-delta}
|z^{k+1}_j| \ge \delta := \min\limits_{i\notin T} \delta_i > 0,\ \mbox{if}\ z^{k+1}_j \neq 0,
\end{align}
where $T = \{i\, \mid \, l_i = u_i = 0\}$, and for $i\notin T$, $\delta_i$ is diffined by
\begin{align*}
\delta_i =
\begin{cases}
\min (u_i,\ \sqrt{2\lambda\tau}),\ &\mbox{if}\ l_i = 0, \\
\min (l_i,\ \sqrt{2\lambda\tau}),\ &\mbox{if}\ u_i = 0, \\
\min (l_i,\ u_i,\ \sqrt{2\lambda\tau}),\ &\mbox{otherwise}.
\end{cases}
\end{align*}
\end{lemma}
The following lemma shows that for the sequence $\{x^k\}$, the magnitude of any nonzero components $x^k_i$ cannot be too small for $k\ge 1$.
\begin{lemma}\label{lem-notsmall}
Let $\{x^k\}$ be generated by Algorithm \ref{alg-PIHT-SGB} and $\delta > 0$ be defined as in \eqref{eq-delta}. For all $k > 0$, it holds that
\begin{description}
\item{(i)} $|x^k_j| > \delta$ for $j \in {\rm{supp}} (x^k)$.

\item{(ii)} $\|x^{k+1} - x^{k}\| > \delta$ whenever ${\rm{supp}} (x^k)\neq {\rm{supp}} (x^{k+1})$.
\end{description}
\end{lemma}
\begin{proof}
(i) By Lemma \ref{lem-notsmall10}, together with
\eqref{group-iteration-2} and \eqref{group-iteration-3} in Algorithm \ref{alg-PIHT-SGB}, we can easily see that (i) holds.

(ii) Suppose that $\text{supp} (x^k)\neq \text{supp} (x^{k+1})$. {For simplicity}, let $i\in [n]$ such that $x^{k+1}_i = 0$ and $x^{k}_i \neq 0$. By (i), it holds that
$$
\|x^{k+1} - x^{k}\| \ge |x^{k+1}_i-x^{k}_i| > \delta.
$$
The proof is complete.
\end{proof}
The following {lemma} displays that the support sets of $\{x^k\}$ remain unchanged when $k$ is sufficiently large, which is important for our convergence analysis.
\begin{lemma}\label{lem-unchange}
Let $f$ be strongly smooth with constant $L>0$. Let $\frac{1}{\tau} >L$ and $\{x^k\}$ be generated by Algorithm \ref{alg-PIHT-SGB}. It holds that
\begin{enumerate}
\item [(i)] $\phi(x^{k}) - \phi(x^{k+1})\ge \frac{1/\tau - L}{2} \|x^{k+1} - x^{k}\|^2$ .

\item [(ii)] $\{\phi(x^k)\}$ is non-increasing and convergent, and $\|x^{k+1}-x^{k} \| \to 0$ as $k\to \infty$.
\item [(iii)] Moreover, {there exists $K>0$ such that} ${\rm{supp}} (x^k) = {\rm{supp}} (x^{k+1})$ {for all $k \ge K$}.
\end{enumerate}
\end{lemma}
\begin{proof}
{(i) By Proposition \ref{pro-proximalgroup}, and Lemma \ref{lem-Fdecreasing}, we can easily  obtain (i).}

(ii) Denote $\eta = \frac{1/\tau - L}{2}$ and {$d^k = x^{k+1}-x^k$}. By result (i), we obtain that
\begin{align*}
\phi(x^{k+1} )- \phi(x^{k} ) \leq - \eta \|d^{k} \|^{2} .
\end{align*}
This implies that $\{\phi(x^k)\}$ is non-increasing and convergent due to the fact that $f$ is bounded from below. Moreover, it holds that
\begin{align*}
\sum_{k=0}^{\infty} \eta \|d^{k} \|^{2}  \leq \sum_{k=0}^{\infty}\left(\phi(x^{k} )-\phi(x^{k+1} )\right)  =\phi(x^{0} )- \lim\limits_{k\to \infty} \phi(x^k) < +\infty .
\end{align*}
Consequently, it holds that $\|x^{k+1}-x^{k} \| \to 0$, as $k\to \infty$.

(iii) Together with Lemma \ref{lem-notsmall} (ii), we obtain that $\text{supp} (x^k) = \text{supp} (x^{k+1})$ holds for sufficiently large $k$.
\end{proof}
\subsection{Convergence and complexity analysis}
We are ready to show that the sequence $\{x^k\}$ converges to a local minimizer of \eqref{GBL0}.
\begin{theorem}\label{the-globalconvergence}
Let $f$ be strongly smooth with constant $L>0$. Let $\frac{1}{\tau} >L$ and $\{x^k\}$ be generated by Algorithm \ref{alg-PIHT-SGB}. {The following results hold.}
\begin{description}
\item{(i)} Any accumulation point $x^*$ of $\{x^k\}$ is a $\tau$-stationary point of \eqref{PBL0}.
\item{(ii)} If $f$ is convex, {then $\{x^k\}$ converges to a local minimizer $x^*$} of \eqref{PBL0} and ${\rm{supp}}(x^k) \to {\rm{supp}}(x^*)$.
\end{description}
\end{theorem}
\begin{proof}
(i). {Let $\{x^{s_i}\}$ be the convergent subsequence of $\{x^k\}$ that converges to $x^*$. Since $x^{s_i} \to x^*$ and $\|x^{k+1} - x^k\| \to 0$ from Lemma \ref{lem-unchange}, we have $x^{s_{i+1}} \to x^*$}. Note that $x^{s_i + 1} \in {\rm{Prox}}_{\tau p(\cdot)}(s_{\tau}(x^{s_i}))$ and \eqref{eq-Def-proximal}. For any $y\in \Omega$, it holds that
\begin{align*}
&\ \ \ \ \frac{1}{2}\|x^{s_i + 1} - s_{\tau}(x^{s_i})\|^2 + \tau \lambda\|x^{s_i + 1}\|_0 + \tau \mu\|x^{s_i + 1}\|_{2,0} \\
&\le \frac{1}{2}\|y - s_{\tau}(x^{s_i})\|^2 + \tau \lambda\|y\|_0 + \tau \mu\|y\|_{2,0}.
\end{align*}
Letting $i\to \infty$,  we obtain that
\begin{align*}
&\ \ \ \ \frac{1}{2}\|x^* - s_{\tau}(x^*)\|^2 + \tau \lambda\|x^*\|_0 + \tau \mu\|x^*\|_{2,0} \le \frac{1}{2}\|y - s_{\tau}(x^{*})\|^2 + \tau \lambda\|y\|_0 + \tau \mu\|y\|_{2,0}.
\end{align*}
By the closedness of $\Omega$, it holds that $x^* \in \Omega$. Therefore, $x^{*} \in {\rm{Prox}}_{\tau p(\cdot)}(s_{\tau}(x^{*}))$, implying that $x^*$ is a $\tau$-stationary point.

(ii). {It follows from Lemma \ref{lem-unchange} that there exist $K>0$ and $\mathcal{S}\subset \{1,\cdots,n\}$ such that \begin{align}\label{eq-theglobal-1}
{\rm{supp}} (x^k) = \mathcal{S},
\end{align}
for all $k\ge K$. Together with \eqref{eq-pro-1} and \eqref{eq-groupproximal-equi}, we obtain that
\begin{align*}
x^{k+1} \in \arg \min\limits_{x \in \mathbb{R}^n} \Psi_i(x;\, s_{\tau}(x^k)) := \frac{1}{2}\|x - s_{\tau}(x^k)\|^2 + \delta_{\Omega_{\mathcal{S}}}(x),
\end{align*}
where $\Omega_{\mathcal{S}} := \{x\in \Omega \, \mid \, {\rm{supp}}(x)\subseteq \mathcal{S}\}$. This shows that $\{x^k\}$ is a sequence generated by the projected gradient method for the following constrained problem
\begin{align*}
\min\limits_{x\in \Omega_{\mathcal{S}}} f(x).
\end{align*}
By \cite[Theorem 2.2]{lu2014iterative}, we obtain that $x^k \to x^*$, where
\begin{align*}
x^* \in \arg \min\limits_{x\in \Omega_{\mathcal{S}}} f(x).
\end{align*}
By definition, $x^*$ is a SO point of \eqref{PBL0}. With Lemma \ref{lem-GBL0GSO}, $x^*$ is local minimizer of \eqref{PBL0}.}

{
From Lemma \ref{lem-notsmall}, we know that $|x^k_i| > \delta$ for sufficiently large $k$ and $i\in \mathcal{S}$. Therefore, {it is not hard} to see that $|x^*_i| \ge \delta$ for $i\in \mathcal{S}$ and $|x^*_i| =0$ for $i\notin \mathcal{S}$ by the fact that $x^k \to x^*$. Consequently, $\mathcal{S} = {\rm{supp}}(x^*)$ for sufficiently large $k$. Together with \eqref{eq-theglobal-1}, we obtain that ${\rm{supp}}(x^k) \to {\rm{supp}}(x^*)$. The proof is complete.}
\end{proof}

\begin{theorem}\label{the-supp-complexity}
{Let $f$ be {strongly smooth} with constant $L > 0$. Also we assume that $f$ is convex} {and $\frac{1}{\tau} > L$}. Let $\{x^k\}$ be generated by Algorithm \ref{alg-PIHT-SGB}. The following results hold.
\begin{enumerate}
  \item [(i)] $\{x^{k}\}$ is a convergent sequence.
  \item [(ii)] {Let $ x^*$ be the point such that $x^k \to x^*$.} Then, the number of changes of ${\rm{supp}} (x^k)$ is at most $\frac{2(\phi(x^0) - \phi^*)}{\delta^2(\frac{1}{\tau} - L)}$, {where $\phi^* $ denotes $ \phi(x^*)$}.
\end{enumerate}
\end{theorem}
\begin{proof}
(i) directly follows from
Theorem \ref{the-globalconvergence}. To show (ii), by Lemma \ref{lem-unchange}, denote the number of changes of $\text{supp} (x^k)$ is $C > 0$. Without loss of generality, suppose $\text{supp} (x^k)$ only changes at $k_{j+1}$. That is, $\text{supp} ({x^{k_j}}) \neq \text{supp} (x^{k_{j+1}})$ for each $j\in \{1,\, \cdots,\, C\}$. By Lemma \ref{lem-notsmall}, we obtain that
$$ \|x^{k_j+1}-x^{k_j} \| > \delta ,\  j\in \{1,\, \cdots,\, C\},$$
which, together with result (i) of Lemma \ref{lem-unchange}, implies that
\begin{align}\label{eq-Fdecreasing-2}
\phi(x^{k_j}) - \phi(x^{k_{j}+1})\ge \delta^2 \frac{\frac{1}{\tau} - L}{2} , \ j\in \{1,\, \cdots,\, C\}.
\end{align}
Summing up these inequalities and using the monotonicity of $\{\phi(x^k)\}$, we obtain that
\begin{align*}
C\delta^2\frac{\frac{1}{\tau} - L}{2} \le \phi(x^{k_1}) - \phi(x^{k_{C}+1}) \le \phi(x^0) - \phi^*,
\end{align*}
which gives
\begin{align}\label{eq-totalnum}
C \le \frac{2(\phi(x^0) - \phi^*)}{\delta^2 (\frac{1}{\tau} - L)}.
\end{align}
The proof is complete.
\end{proof}

Define
\begin{align}
\Gamma(x^*) &=\left(s_{\tau}(x^*)\right)^2 - \left( d_{\tau}(x^*)\right)^2,\nonumber \\
\rho_i &= \left| (s_{\tau}(x^*))_i \right| + \left| (d_{\tau}(x^*))_i \right|,\ \text{for}\ i\in [n], \nonumber \\
\alpha &= \min\limits_{I\subseteq [n]}\left\{ \min\limits_{i} \left| \left(\Gamma(x^*) \right)_i - 2\tau \lambda \right| \, \big| \, x^* \in \arg\min\limits_{x\in \mathbb{R}^n}\{f(x)\, \mid \, x\in \Omega ,\ x_{I} = 0\} \right\} , \label{eq-alpha}\\
\beta &= \max\limits_{I\subseteq [n]}\left\{ \max\limits_{i}\rho_i \, \big| \, x^* \in \arg\min\limits_{x\in \mathbb{R}^n}\{f(x)\, \mid \, x\in \Omega , x_{I} = 0\} \right\}. \label{eq-beta}
\end{align}
We now establish the iteration complexity for the PIHT-SGB method. {The following theorem shares the similar conclusion as that in {\rm{\cite[Theorem 3.5 (ii)]{lu2014iterative}}}. The proof is almost identical to that in \cite[Theorem 3.5 (ii)]{lu2014iterative}. So we omit the proof here.}
\begin{theorem}\label{the-complexity}
Assume that $f$ is a strongly smooth and strongly convex function with constant $L$, $\ell > 0$. Let $\alpha > 0$ in \eqref{eq-alpha}. Suppose $\frac{1}{\tau} > L$ such that $\alpha > 0$. Let $\{x^k\}$ be generated by Algorithm \ref{alg-PIHT-SGB}, $x^* = \lim\limits_{k\to \infty} x^k$, $\phi^* = \phi(x^k)$. Then, for any given $\epsilon >0$, the total iterations number by Algorithm \ref{alg-PIHT-SGB} for finding a $\epsilon$-local-optimal solution $x_{\epsilon} \in \Omega$ satisfying ${\rm{supp}}(x_{\epsilon}) = {\rm{supp}}(x^{*})$ and $ \phi(x_{\epsilon}) \le \phi^* + \epsilon $ is at most $\frac{2}{\tau\ell}\log \frac{\theta}{\epsilon}$, where
\begin{align}
\theta &= 2^{\frac{\omega + 3}{2}}(\phi(x^0) - \phi^*) ,\nonumber \\
\omega &= \max\limits_{t}\left\{ (d-2b)t - bt^2\, \bigg| \, 0\le t\le \frac{2(\phi(x^0) - \phi^*)}{\delta^2 (\frac{1}{\tau} - L)}   \right\} ,\label{eq-omega}\\
b &= \frac{\delta^2(\frac{1}{\tau} - L)}{2(\phi(x^0) - \underline{\phi}^*)},\ \gamma = \ell\left(\sqrt{2\alpha + \beta^2} - \beta\right)^2/32,  \label{eq-gamma}\\
d &= 2\log (\phi(x^0) - \underline{\phi}^*) + 4 - 2\log \gamma + b. \nonumber
\end{align}
\end{theorem}

\section{Numerical Results}
In this section, we report the numerical results of our proposed PIHT-SGB in this paper. To demonstrate the advantages of our algorithm, we conducted a comparative analysis with PIHT \cite{lu2014iterative} (proximal iterative hard thresholding methods) , DCGL (Difference of-Convex algorithm for Sparse Group $\ell_0$ problem) \cite{doi:10.1137/21M1443455}. The algorithms are implemented in MATLAB R2020a on a personal laptop with AMD Ryzen 7 5800H with Radeon Graphics 3.20 GHz CPU and 16GB memory.

For DCGL, we use the same parameter setting as it in \cite{doi:10.1137/21M1443455}.
For PIHT-SGB, we set the stopping conditions as follows
\begin{align}\label{eq-stop}
\frac{\|x^k - x^{k-1}\|}{\max(1,\|x^k\|) } \le 10^{-6},\ f(x^k) \le \epsilon,\ {\rm{iter}} > 100.
\end{align}
Unless otherwise specified, the default initial point is set to $x^0=(0,\, \cdots,\, 0)^{\top}$. The testing problems {take the form as \eqref{GBL0} with $ f(x) = \frac{1}{2}\|Ax - b\|^2$,} where $b\in \mathbb{R}^m$ is an observation vector and $l$, $u\in \mathbb{R}_{+}^n$ are boundary vectors. The choices of $\epsilon$ and $l, u$ are described in each respective example.

We will report the following results: dimension $n$, number of iterations Iter, computation time (in seconds), and err defined by $err$ = $\|x^k - x^*\|/\|x^*\|$, where $x^*$ is the ground truth solution.

\subsection{Signal recovery}
In this experiment, we verify the efficiency of PIHT-SGB in noisy signal recovery.

\textbf{E1}. \cite{liao2024subspace} We consider the exact recovery $b = Ax^* + \sigma \xi$. $A \in \mathbb{R}^{m\times n}$ are random Gaussian matrix and the columns of $A$ are normalized to have $\ell_2$ norm of 1. We set the number of non-zero components $s = 0.05n$. The noise $\xi$ {is coded as} $\xi = {\rm{randn}}(n,\, 1)$ and $\sigma$ is set to 0.01. The vector $x^*$ is divided into $n_{G} = \frac{n}{w}$ groups with each group containing $w=4$ elements, the group sparsity (number of nonzero groups) is set to $s_{G} = \frac{s}{w}$. The nonzero elements $x^*$ follows a normal distribution with values between 0.1 and 5, which by the following codes
\begin{align*}
x^* = {\rm{zeros}}(n,\, 1),\  x^*_{M} = 0.1 + (5-0.1){\rm{rand}}(s,\, 1),\
\end{align*}
where $M :={\rm{supp}}(x^*)$ which is randomly selected component positions of $s_{G}$ groups. In this case, we set the boundary vector $l=u = 5\times {\rm{ones}}(n,\, 1)$.
For \textbf{E1}, we set $\epsilon = 10^{-20}$ in \eqref{eq-stop}.

\subsubsection{Different dimensions}
In order to display the result of our algorithm intuitively, Fig. \ref{fig1} illustrates the
visualization of the signal recovery effect for different column numbers from $n = 2000$ to $n = 12000$ and row numbers with $m =0.25n$. From Figure \ref{fig1}, we see that the output signals of Algorithm \ref{alg-PIHT-SGB} almost coincide with the ground truth signals. The iterative process
is provided in Fig. \ref{fig2}, when $n=10000$ and $m=0.25n$. One
could see that the function value and $err$ all exhibit a linear descent.

\begin{figure}[htbp]
	\centering
	\includegraphics[width=6cm]{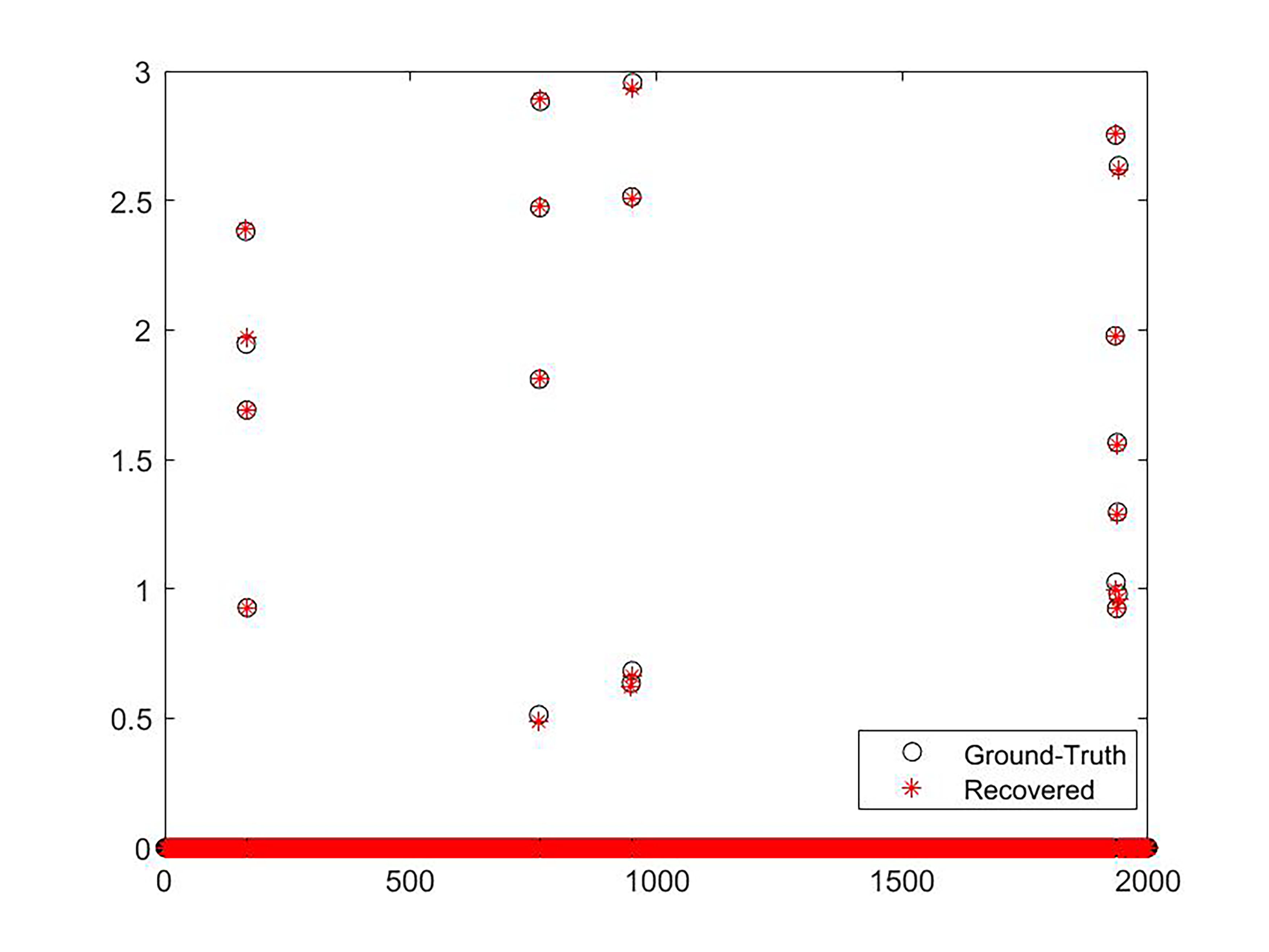
}
    \includegraphics[width=6cm]{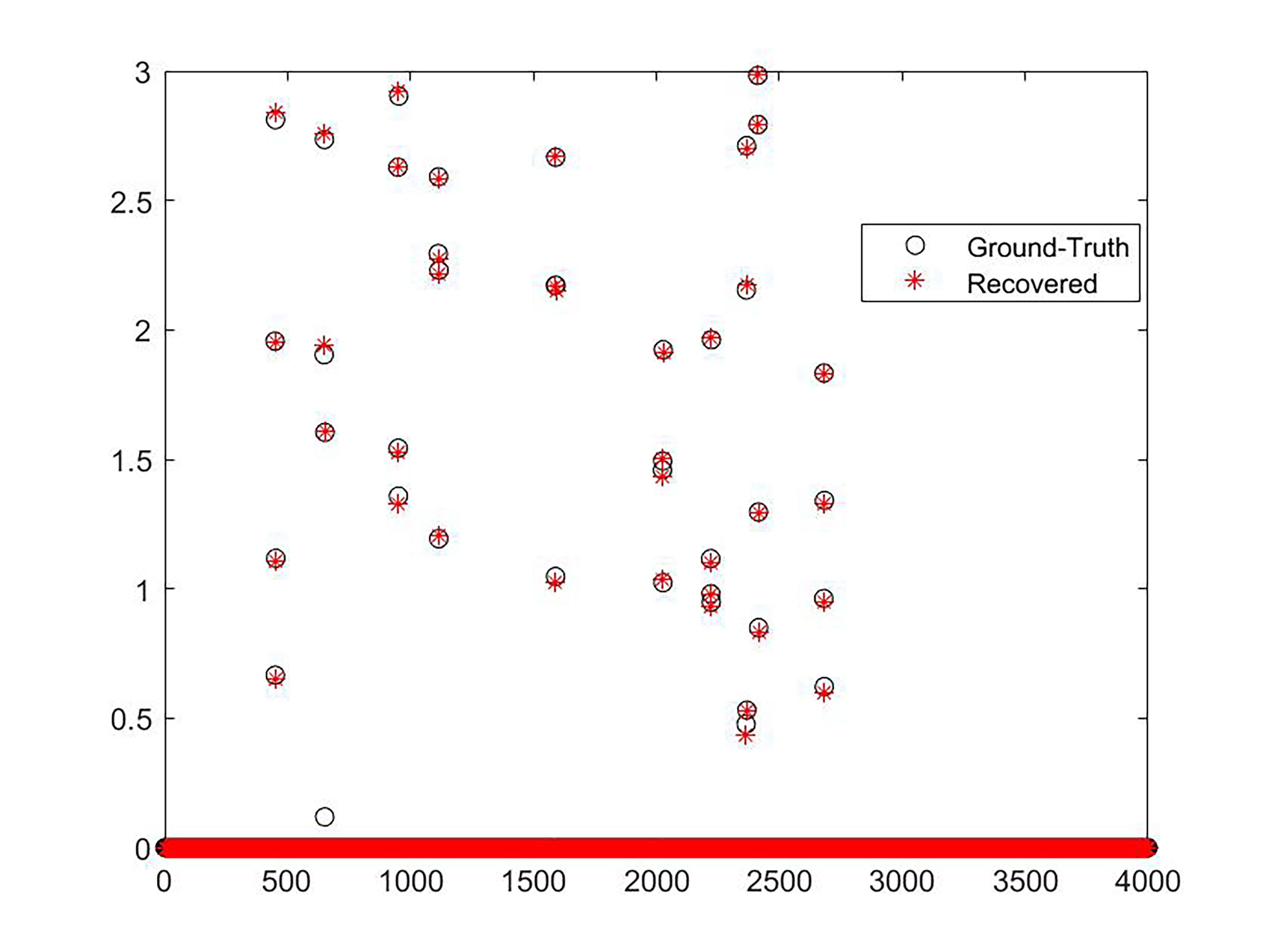
}
	\includegraphics[width=6cm]{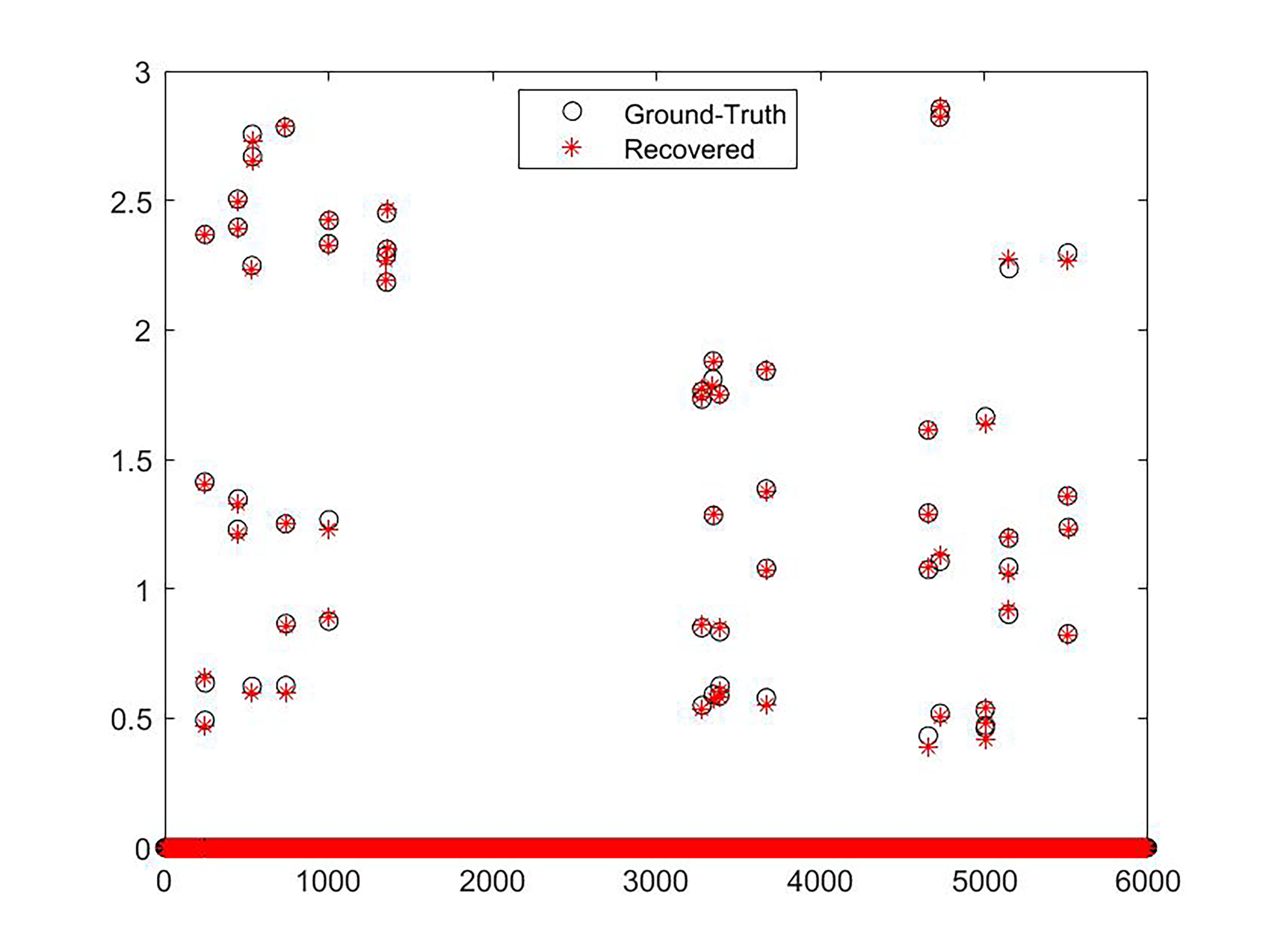
}
    \includegraphics[width=6cm]{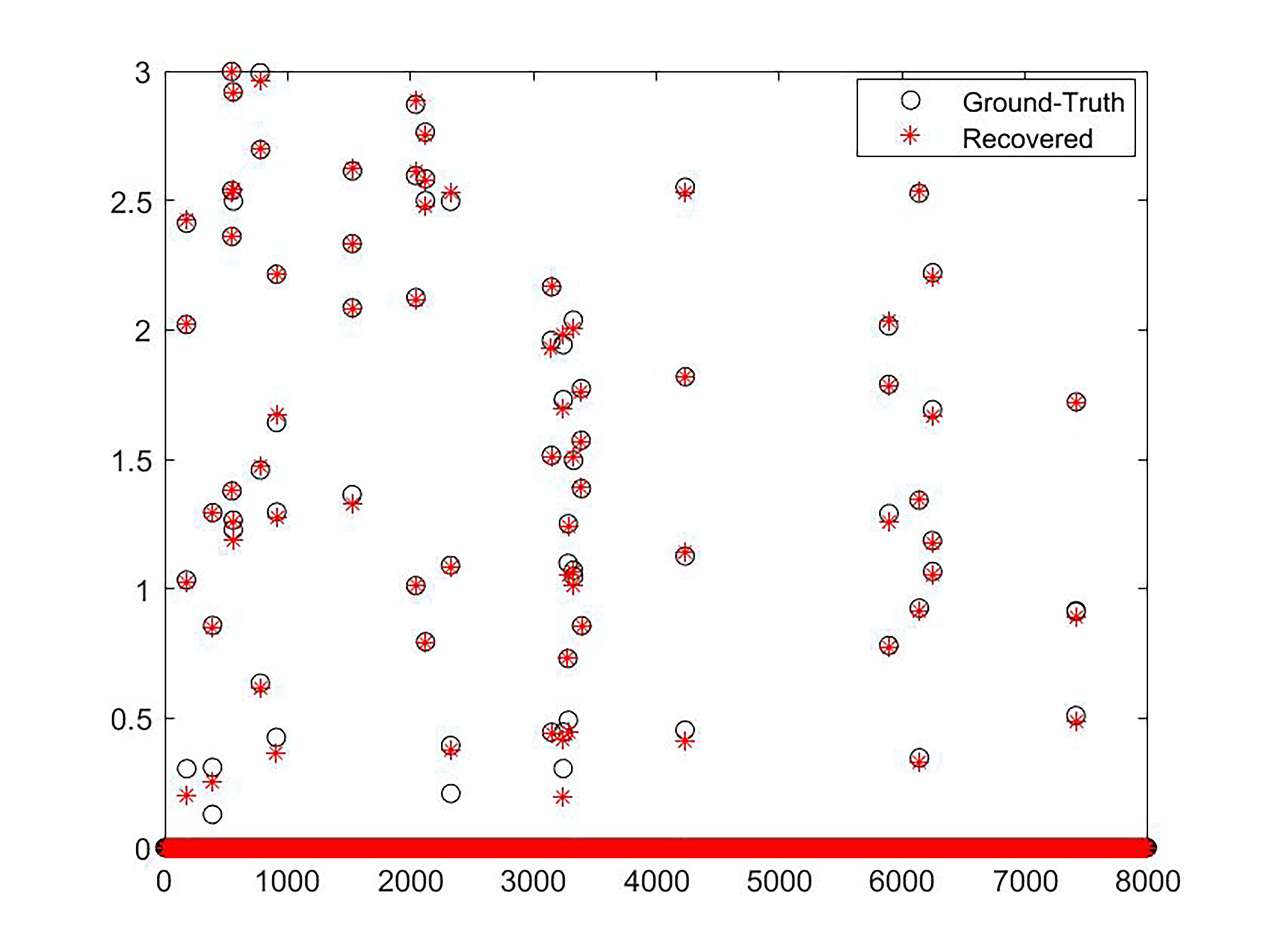
}
	\includegraphics[width=6cm]{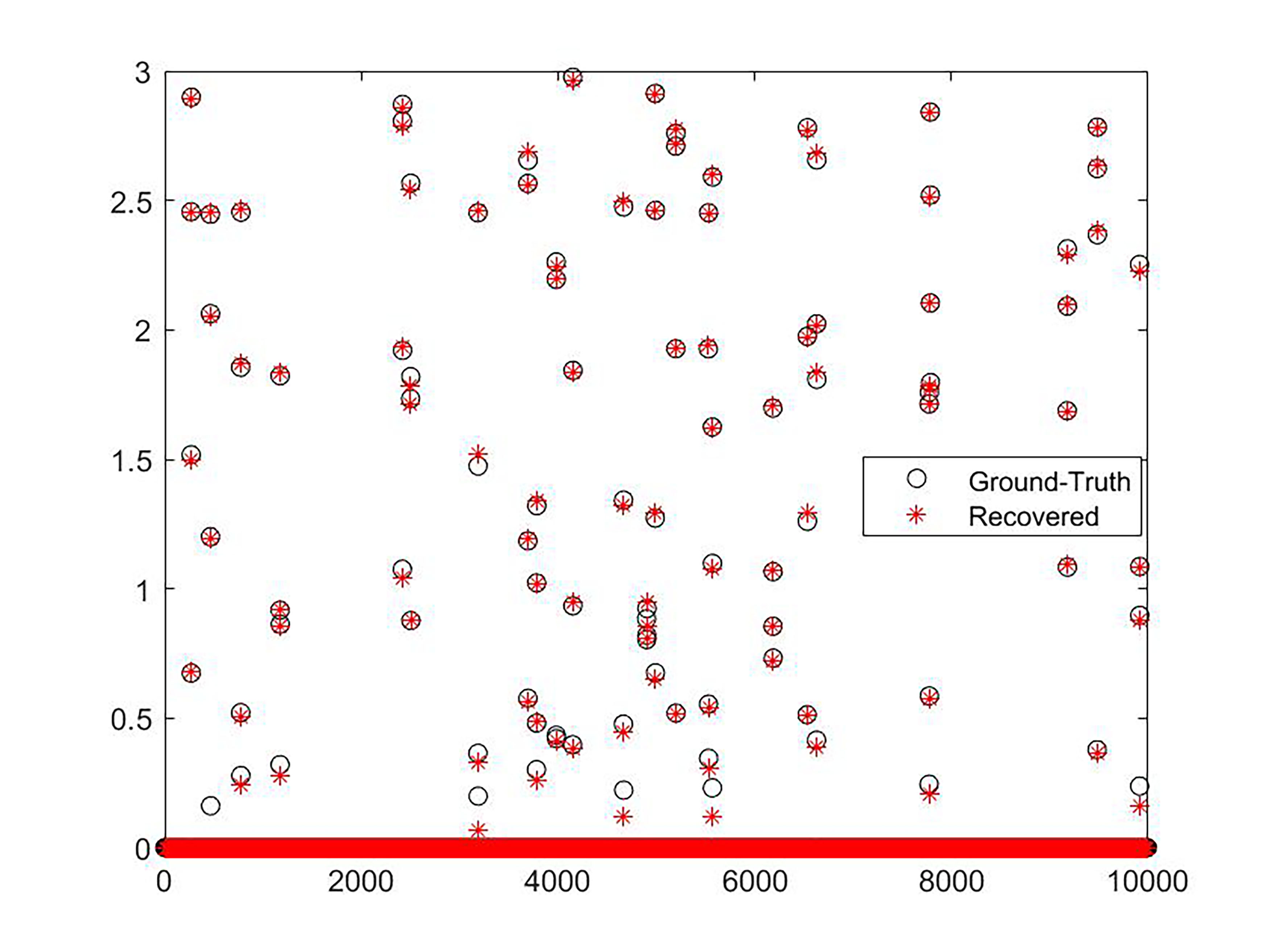
}
    \includegraphics[width=6cm]{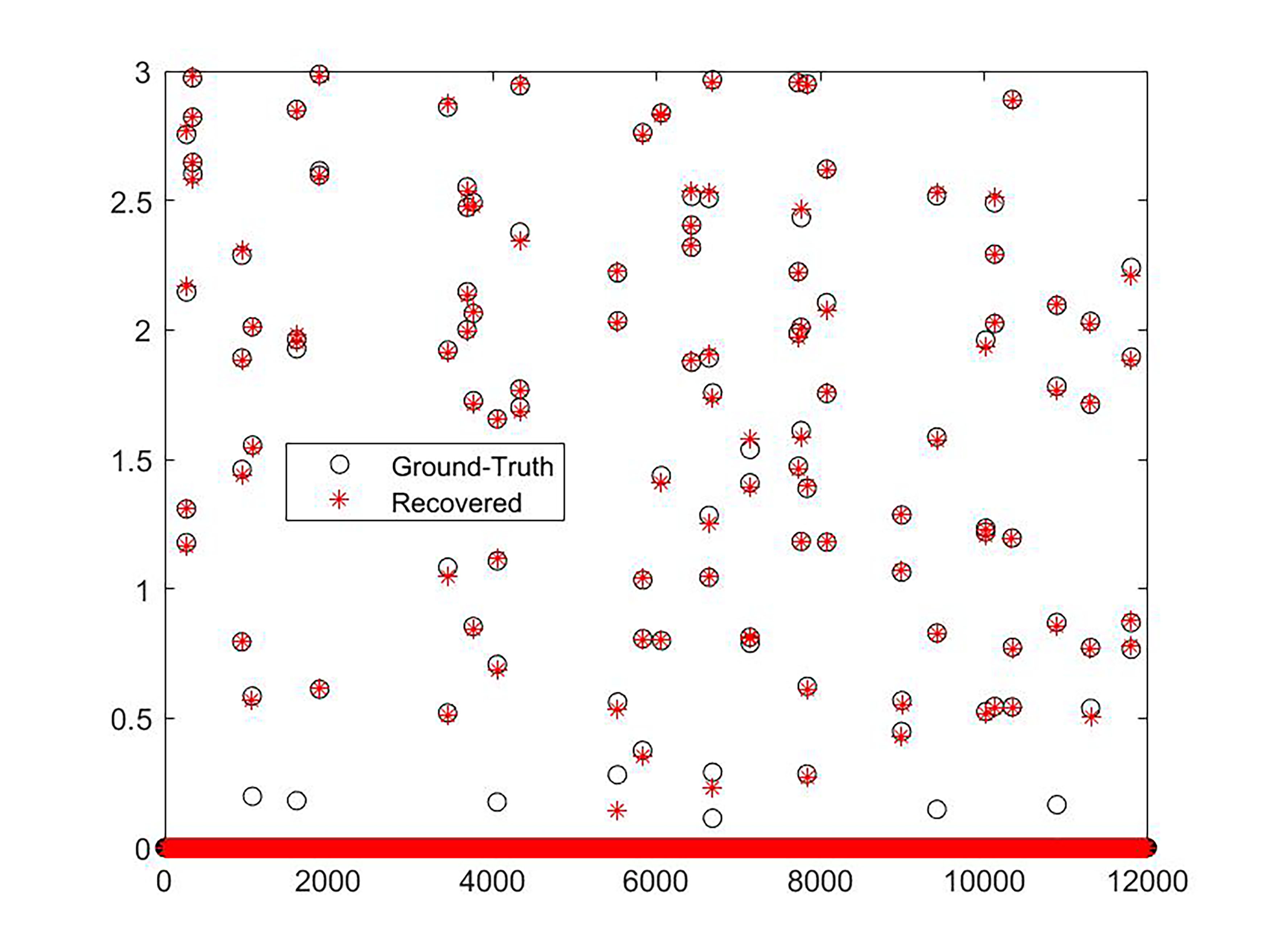
}
	\caption{An illustration of signal restoration by PIHT-SGB with $n=2000 : 12000$ and $m=0.25n$}
	\label{fig1}
\end{figure}

\begin{figure}[htbp]
	\centering
	\includegraphics[width=6cm]{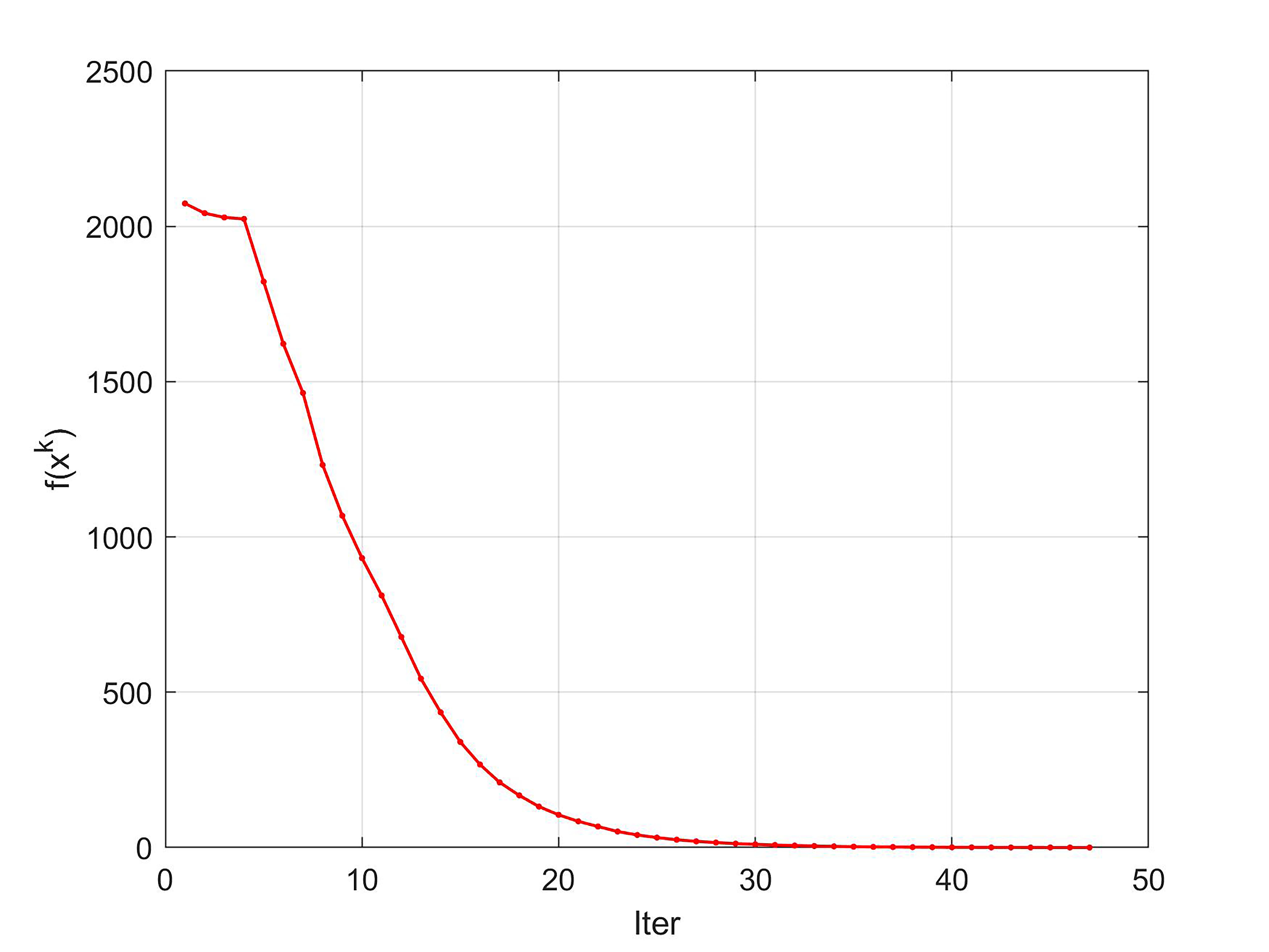
}
\includegraphics[width=6cm]{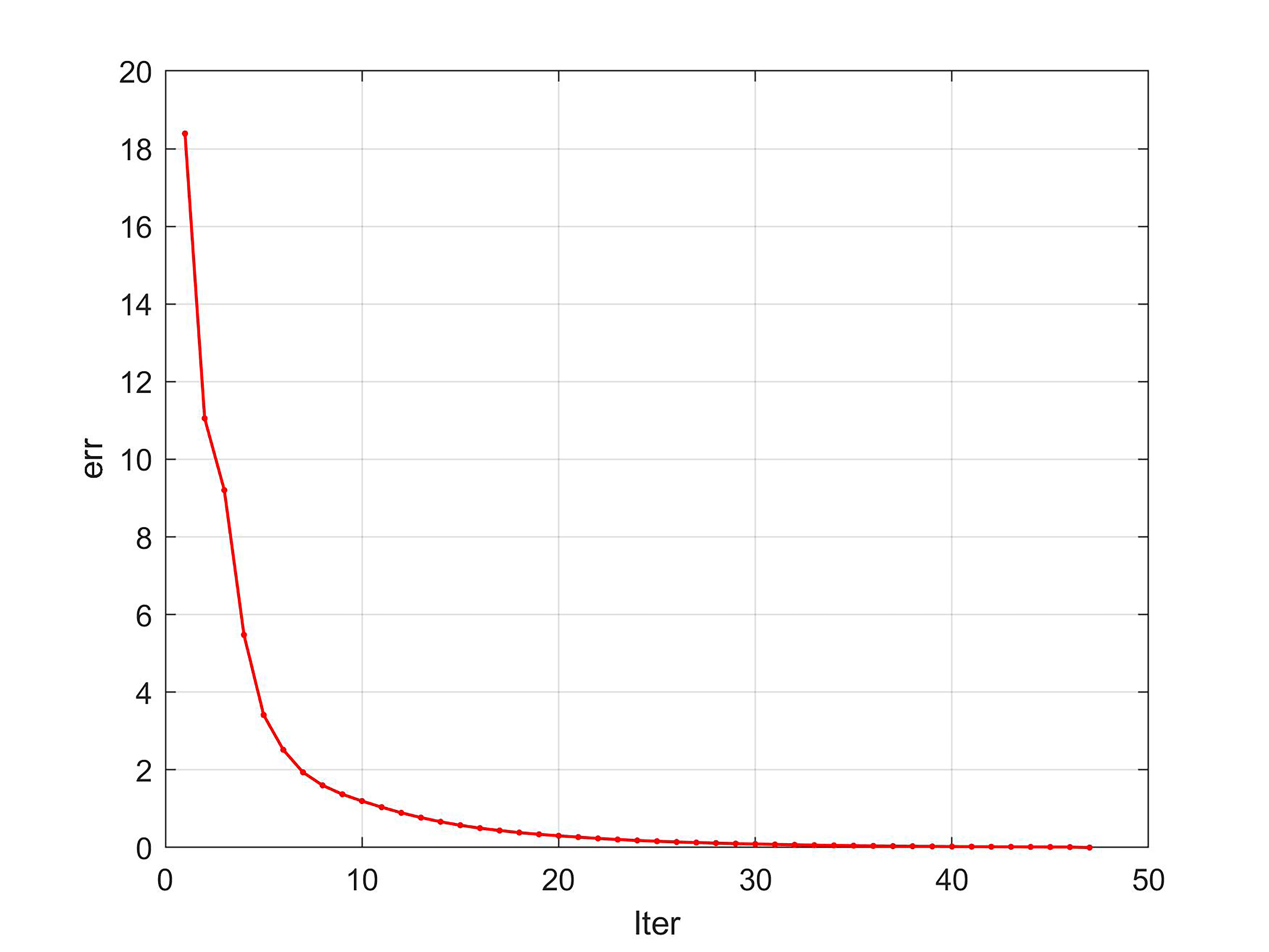
}
	\caption{An illustration of iteration process in PIHT-SGB with $n=10000$ and $m=0.25n$.}
	\label{fig2}
\end{figure}

\begin{table*}[htbp]
\scriptsize
\centering
\caption{\text{Results on noised signal recovery with $m = 0.25n$.}}\label{Tab01}
\begin{tabular*}{\textwidth}{@{\extracolsep\fill}ccccrcccr}
\toprule
\textbf{Algorithm}   &{n} &  iter     &  time (s)   &   err
&{n}  &  iter      &  time (s)   &   err \\
 \midrule
\text{PIHT-SGB} & 5000  & 43  & 0.24  & 1.90e-02 & 7000  & 44  & 0.44  & 1.53e-02 \\
 \midrule
\text{PIHT-SGB} & 9000  & 45  & 0.72  & 1.42e-02 & 11000  & 46  & 1.08  & 1.31e-02 \\
 \midrule
\text{PIHT-SGB} & 13000  & 46  & 1.46  & 1.21e-02 & 15000  & 47  & 1.94  & 1.17e-02 \\
\bottomrule
\end{tabular*}
\end{table*}

\subsubsection{Different initializations}
We choose different initial points $x^0 = \bf{-1},\, \bf{1},\, \bf{0},\, \bf{randn}$ which are {coded} as
\begin{align*}
{\bf{-1}} = -1\times {\rm{ones}}(n,\, 1),\ {\bf{1}} = {\rm{ones}}(n,\, 1),\ {\bf{0}} = {\rm{zeros}}(n,\, 1),\ {\bf{randn}} = {\rm{ randn}}(n,\, 1).
\end{align*}
In this test, we set $\sigma = 0.01$, $m = 0.25n$, $s=0.05n$.
\begin{table*}[htbp]
\scriptsize
\centering
\caption{\text{Results on noised signal recovery with different initial points. }}\label{Tab02}
\begin{tabular*}{\textwidth}{@{\extracolsep\fill}ccccrcccr}
\toprule
\textbf{Initial point}   &{n} &  iter     &  time (s)   &   err
&{n}  &  iter      &  time (s)   &   err \\
 \midrule
\textbf{0} & \multirow{4}{*}{5000}  & 43  & 0.26  & 1.76e-2 & \multirow{4}{*}{7000}  & 44  & 0.48  & 1.61e-2 \\
\textbf{-1} & \multirow{4} {*}{}  & 46  & 0.28  & 3.43e-2 & \multirow{4} {*}{}  & 47  & 0.53  & 2.35e-2 \\
\textbf{+1} & \multirow{4}{*}{}  & 43  & 0.25  & 1.74e-2 & \multirow{4} {*}{}  & 44  & 0.49  & 1.62e-2 \\
\textbf{randn} & \multirow{4}{*}{}  & 45  & 0.27  & 2.56e-2 & \multirow{4} {*}{}  & 46  & 0.51  & 1.95e-2 \\
 \midrule
\textbf{0} & \multirow{4} {*}{9000}  & 45  & 0.75  & 1.40e-2 & \multirow{4} {*}{11000}  & 46  & 1.15  & 1.34e-2 \\
\textbf{-1} & \multirow{4} {*}{}  & 48  & 0.85  & 2.28e-2 & \multirow{4} {*}{}  & 50  & 1.32  & 2.48e-2 \\
\textbf{+1} & \multirow{4} {*}{}  & 45  & 0.78  & 1.41e-2 & \multirow{4} {*}{}  & 46  & 1.16  & 1.32e-2 \\
\textbf{randn} & \multirow{4} {*}{}  & 47  & 0.83  & 1.68e-2 & \multirow{4} {*}{}  & 48  & 1.25  & 1.70e-2 \\
 \midrule
\textbf{0} & \multirow{4} {*}{13000}  & 46  & 1.48  & 1.26e-2 & \multirow{4} {*}{15000}  & 47  & 1.97  & 1.22e-2 \\
\textbf{-1} & \multirow{4} {*}{}  & 50  & 1.68  & 1.94e-2 & \multirow{4} {*}{}  & 50  & 2.23  & 1.80e-2 \\
\textbf{+1} & \multirow{4} {*}{}  & 46  & 1.52  & 1.24e-2 & \multirow{4} {*}{}  & 47  & 2.01  & 1.22e-2 \\
\textbf{randn} & \multirow{4} {*}{}  & 48  & 1.63  & 1.54e-2 & \multirow{4} {*}{}  & 49  & 2.13  & 1.48e-2 \\
\bottomrule
\end{tabular*}
\end{table*}

\begin{figure}[htbp]
	\centering
	\includegraphics[width=6cm]{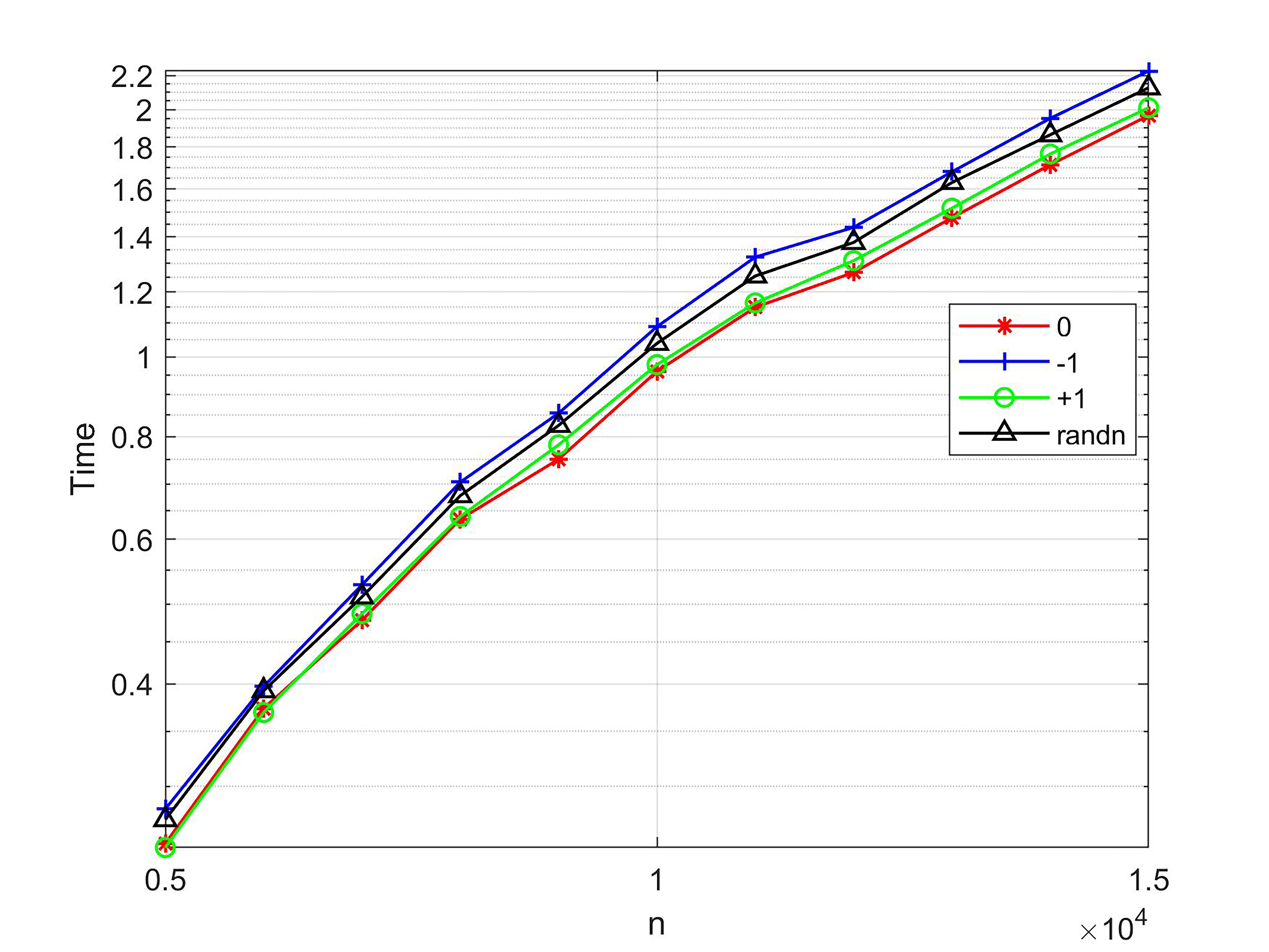
}
\includegraphics[width=6cm]{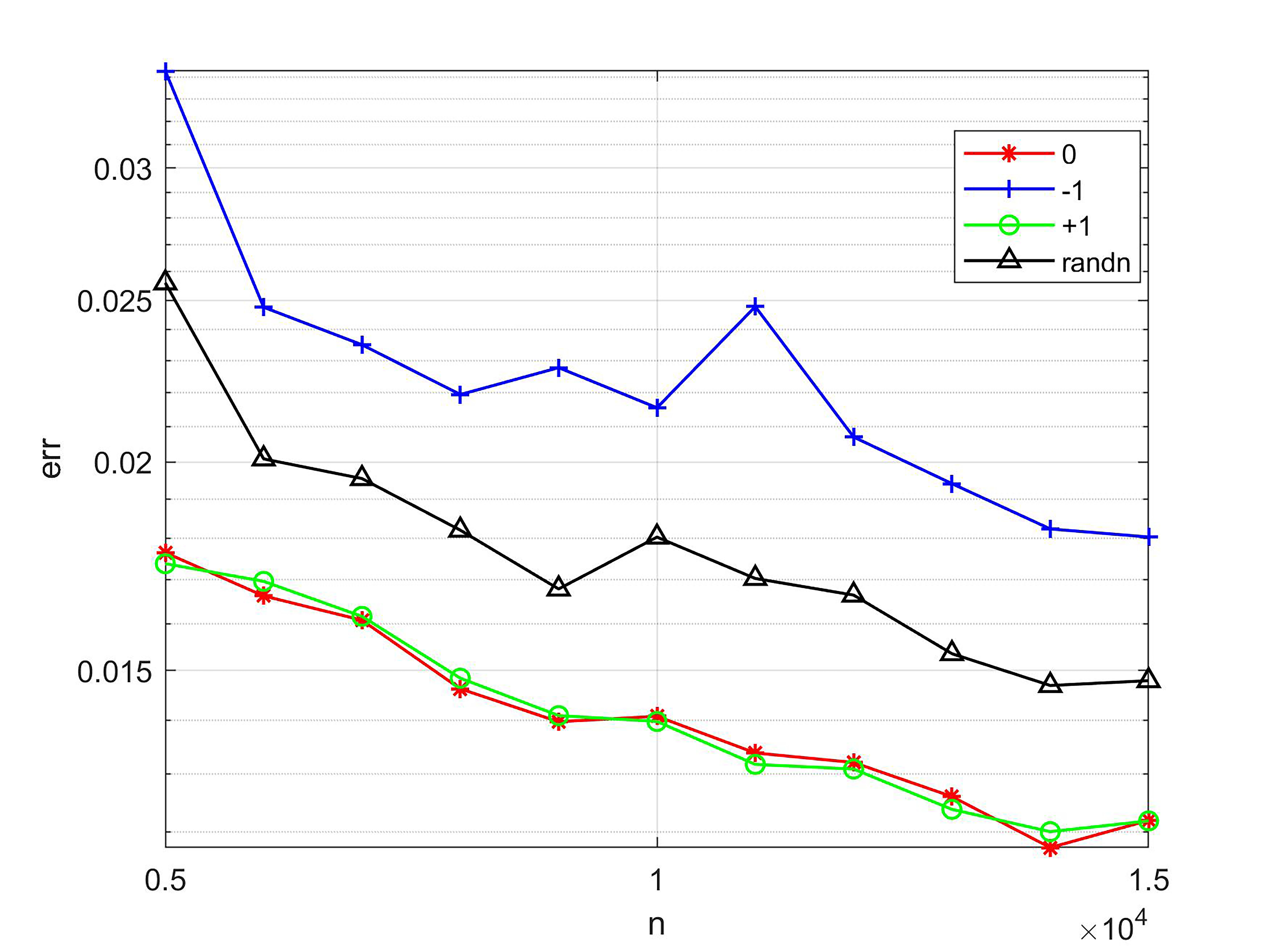
}
	\caption{Signal recovery by PIHT-SGB with different initial points and $m=0.25n$, $s = 0.05n$.}
	\label{fig4}
\end{figure}

From Table \ref{Tab02} and Figure \ref{fig4}, it can be seen that the residuals of the solutions obtained by PIHT-SGB using different initial points are all less than \( 5 \times 10^{-2} \), indicating that the PIHT-SGB algorithm exhibits a certain degree of robustness with respect to the choice of initial points.

\subsubsection{Different boxes}
{This experiment is designed to carry out a comparison: examining the impact of different boxes on PIHT-SGB. These comparisons are assessed by employing the computational time and reconstruction accuracy under varying sparsity levels as the performance metrics. We set different sparsity levels from $0.01n$ to $0.17n$, $w=4$ (once $w$ and $s$ are
given, $n_G$ and $s_G$ is determined), $n=8000$, $m=0.5n$, $\sigma=0.001$. Other parameter settings remain
unchanged. Specifically, we set three boundary vector $u_1 = l_1 = 5\times {\rm{ones}}(n,1)$ for PIHT-SGB (box1), $u_2 = l_2 = 6\times {\rm{ones}}(n,1)$ for PIHT-SGB (box2) and $u_3 = l_3 = 10\times {\rm{ones}}(n,1)$ for PIHT-SGB (box3).  A
trial is regarded as success if $err \le 0.05$.}

\begin{figure}[htbp]
	\centering
	\includegraphics[width=6cm]{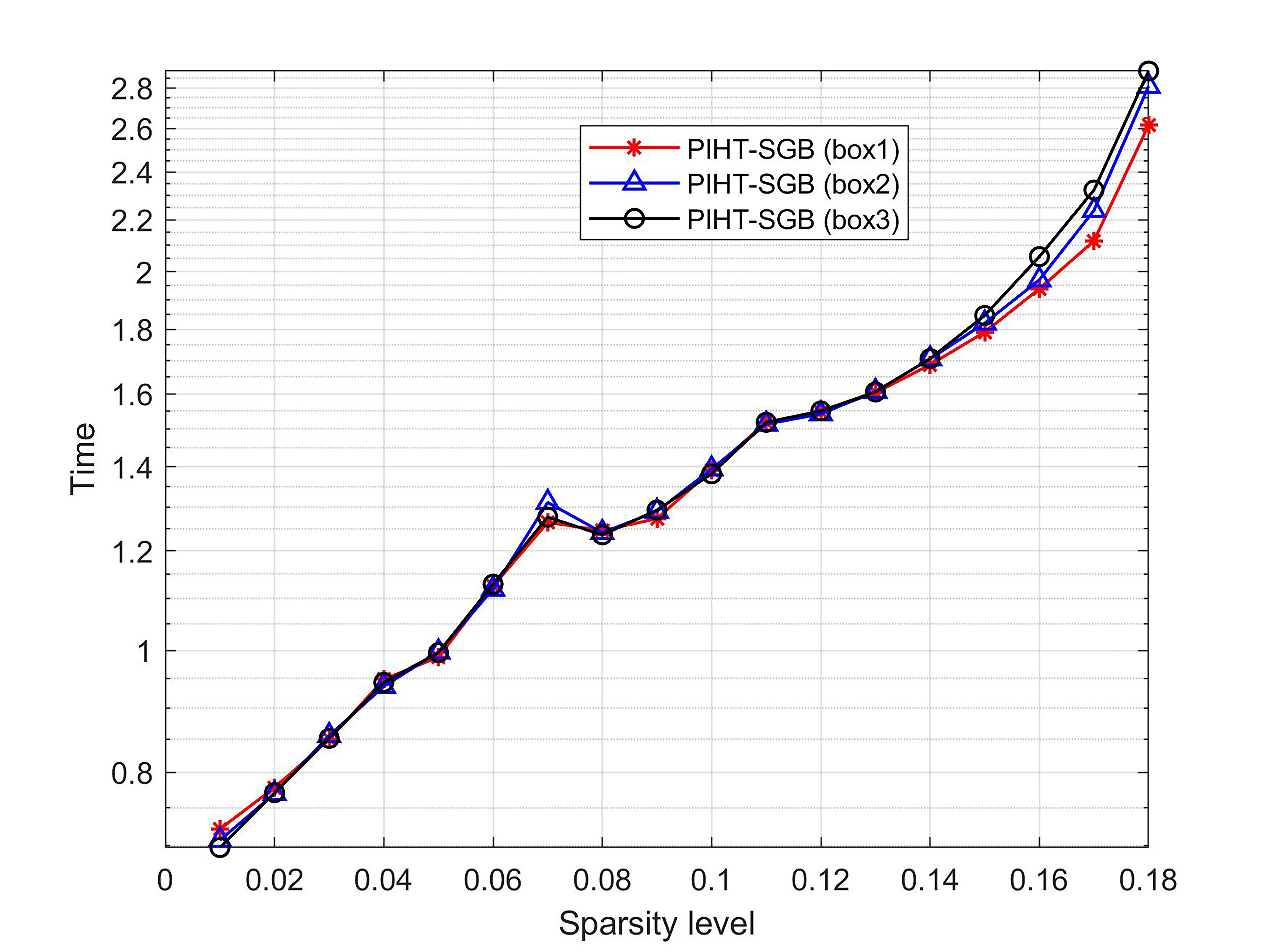
}
\includegraphics[width=6cm]{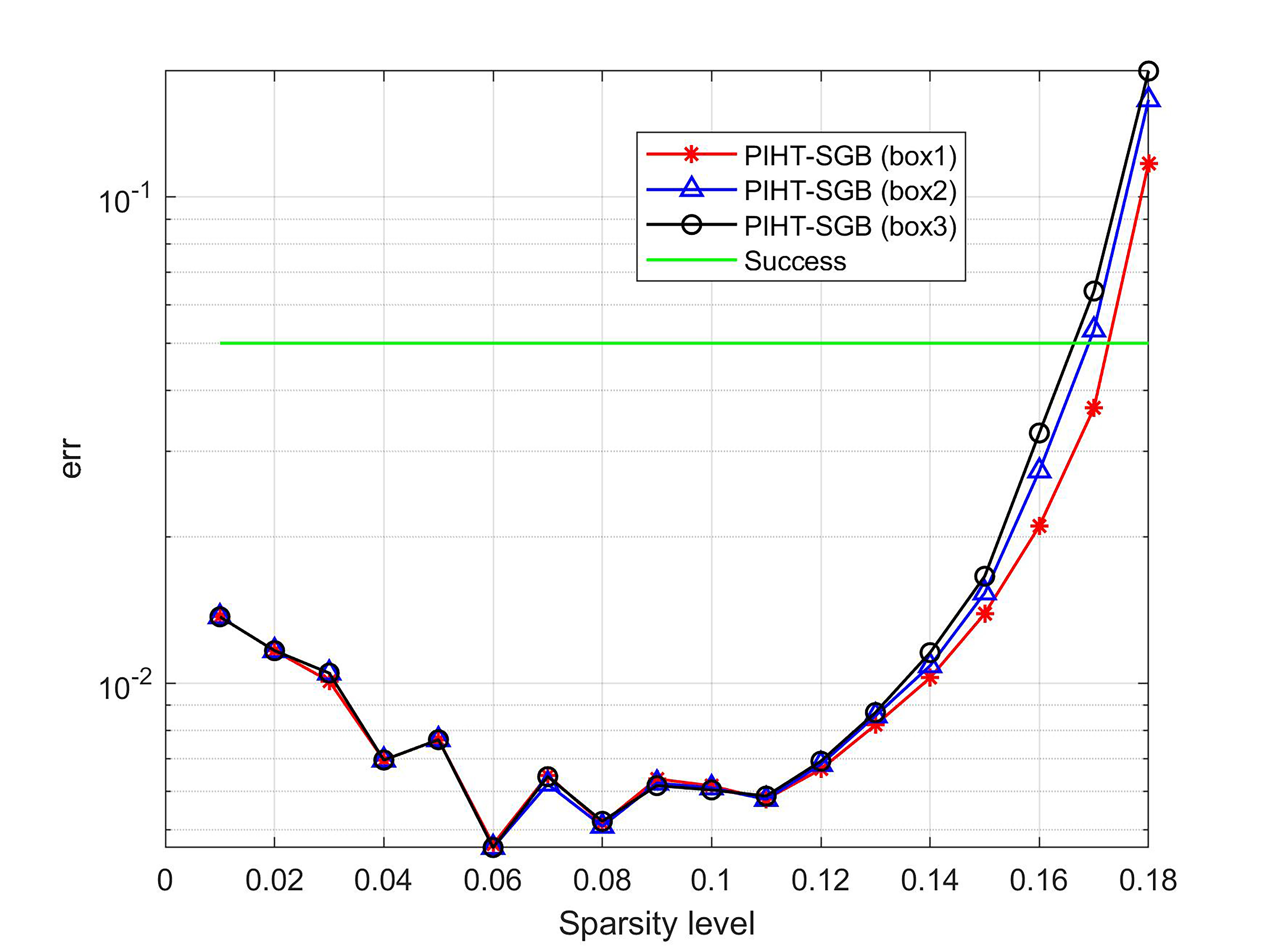
}
	\caption{PIHT-SGB in different boxes.}
	\label{fig5-2}
\end{figure}

{Figure \ref{fig5-2} reveals that the critical role of the box constraint in PIHT-SGB. The results show that constraints closer to the ground truth $x$ (representing more accurate prior knowledge) lead to faster convergence and higher solution accuracy.}
\subsection{Comparison with PIHT}
In this part, we will demonstrate the impact of different group sizes on PIHT and PIHT-SGB under high sparsity conditions, and then evaluate the performance of PIHT and PIHT-SGB with different box constraints across varying sparsity levels.
\subsubsection{Different group sizes }
{In this experiment, based on \textbf{E1} with fixed $n=8000$ and $s=0.14n$, we set different $w \in \{2, 4, 5, 8, 10, 16, 20, 32, 64, 80, 160\}$. As $w$ increases, the number of group decreases. The performance of both PIHT and PIHT-SGB will be evaluated under these varying $w$ settings.}
\begin{figure}[htbp]
	\centering
	\includegraphics[width=6cm]{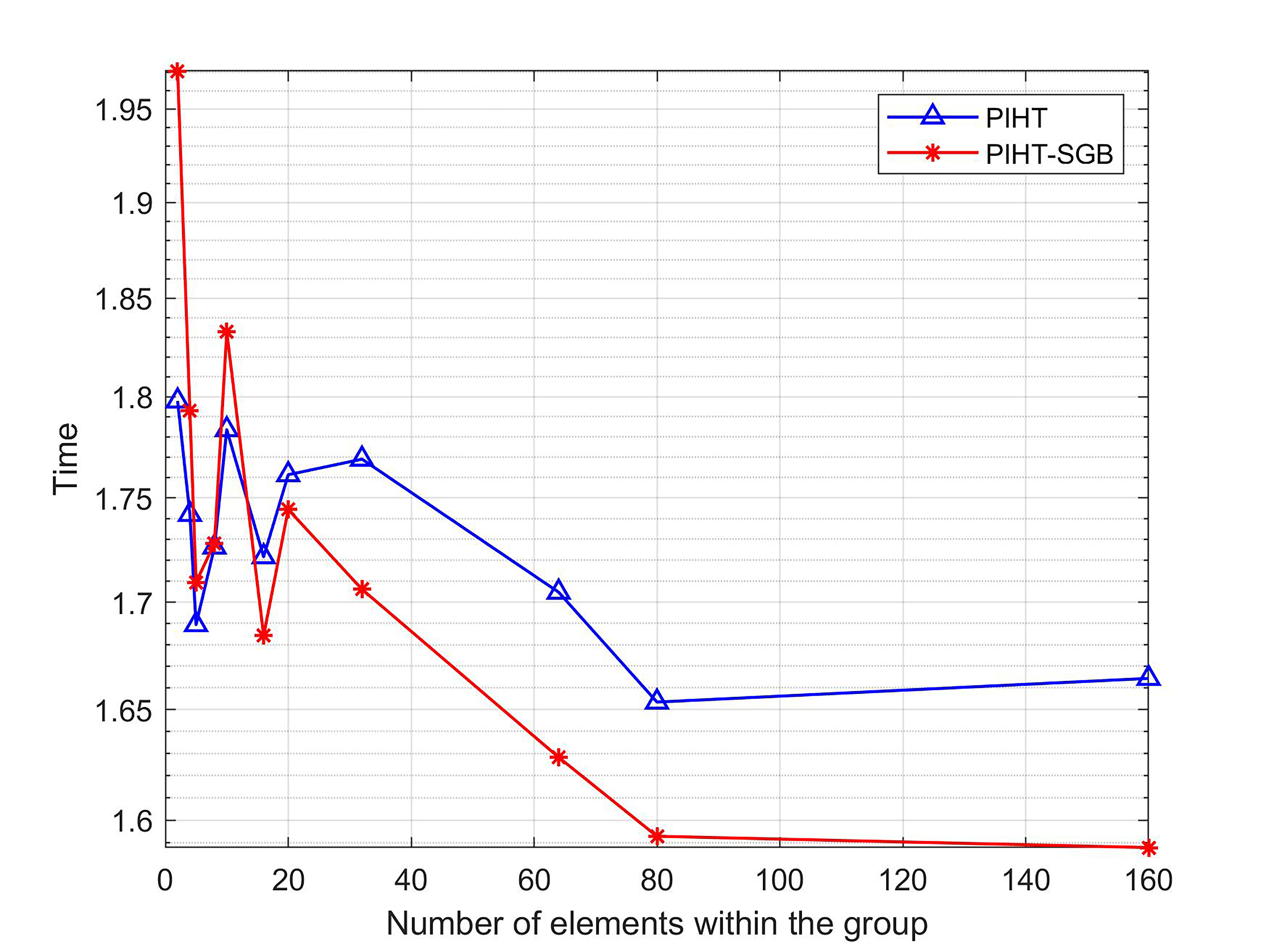
}
\includegraphics[width=6cm]{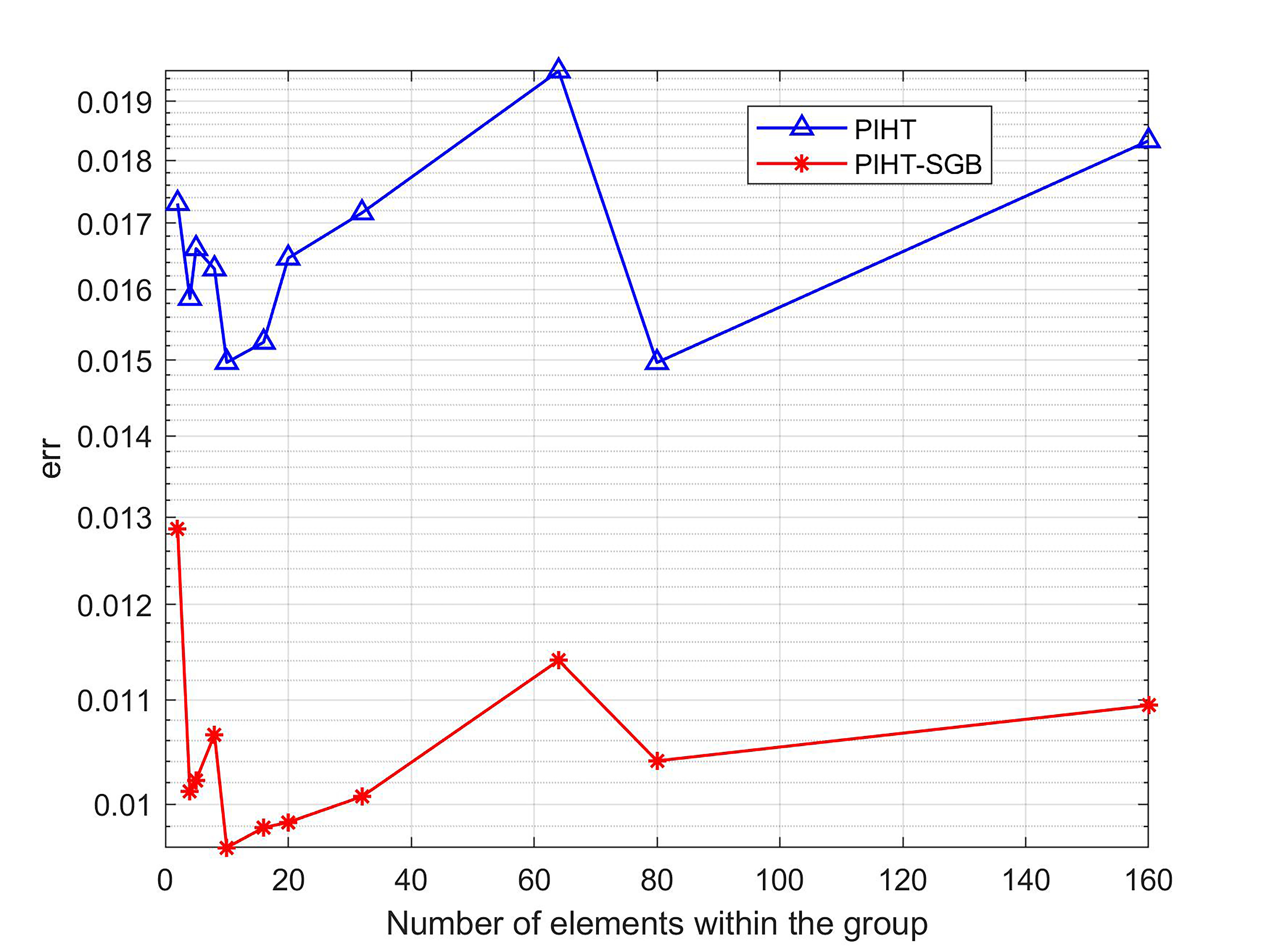
}
	\caption{Comparison with PIHT under different group sizes.}
	\label{fig5-1}
\end{figure}

{As observed from Figure \ref{fig5-1}, when $w$ is relatively small, the computational time of PIHT and PIHT-SGB shows little difference. However, as w increases, PIHT-SGB requires less computational time than PIHT, while consistently achieving a lower recovery error. These results demonstrate that simultaneously considering both group sparsity and element-wise sparsity leads to better reconstruction performance compared to considering element-wise sparsity alone. Furthermore, the number of groups can influence the computational efficiency of the algorithms to some extent.}
\subsubsection{Different sparsity levels }
{In this experiment, based on \textbf{E1} with fixed $n=8000$, $m=0.5n$ and $\sigma = 0.001$, we set different sparsity levels from $0.01n$ to $0.17n$. A
trial is regarded as success if $err \le 0.05$.}
\begin{figure}[htbp]
	\centering
	\includegraphics[width=6cm]{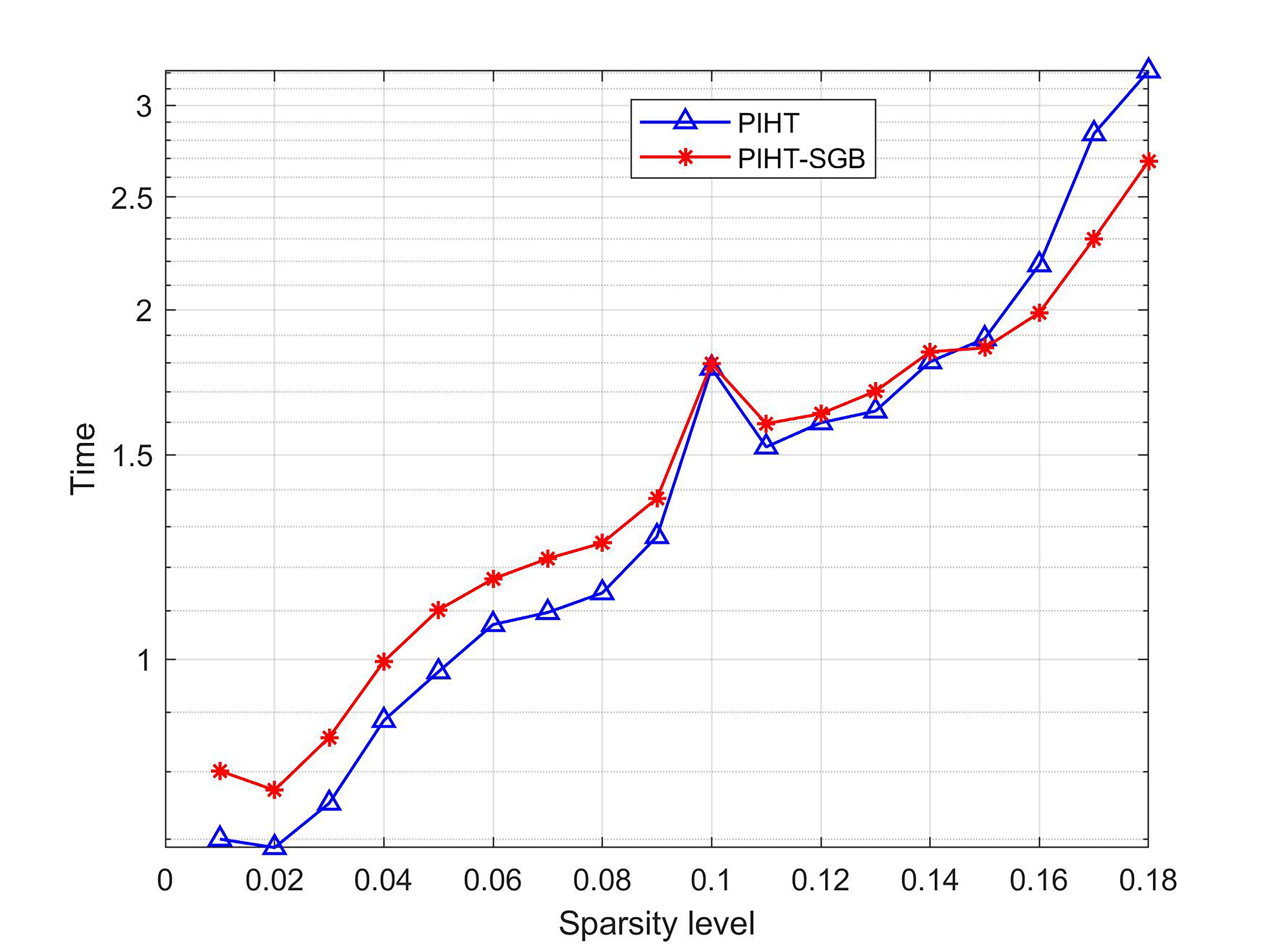}
\includegraphics[width=6cm]{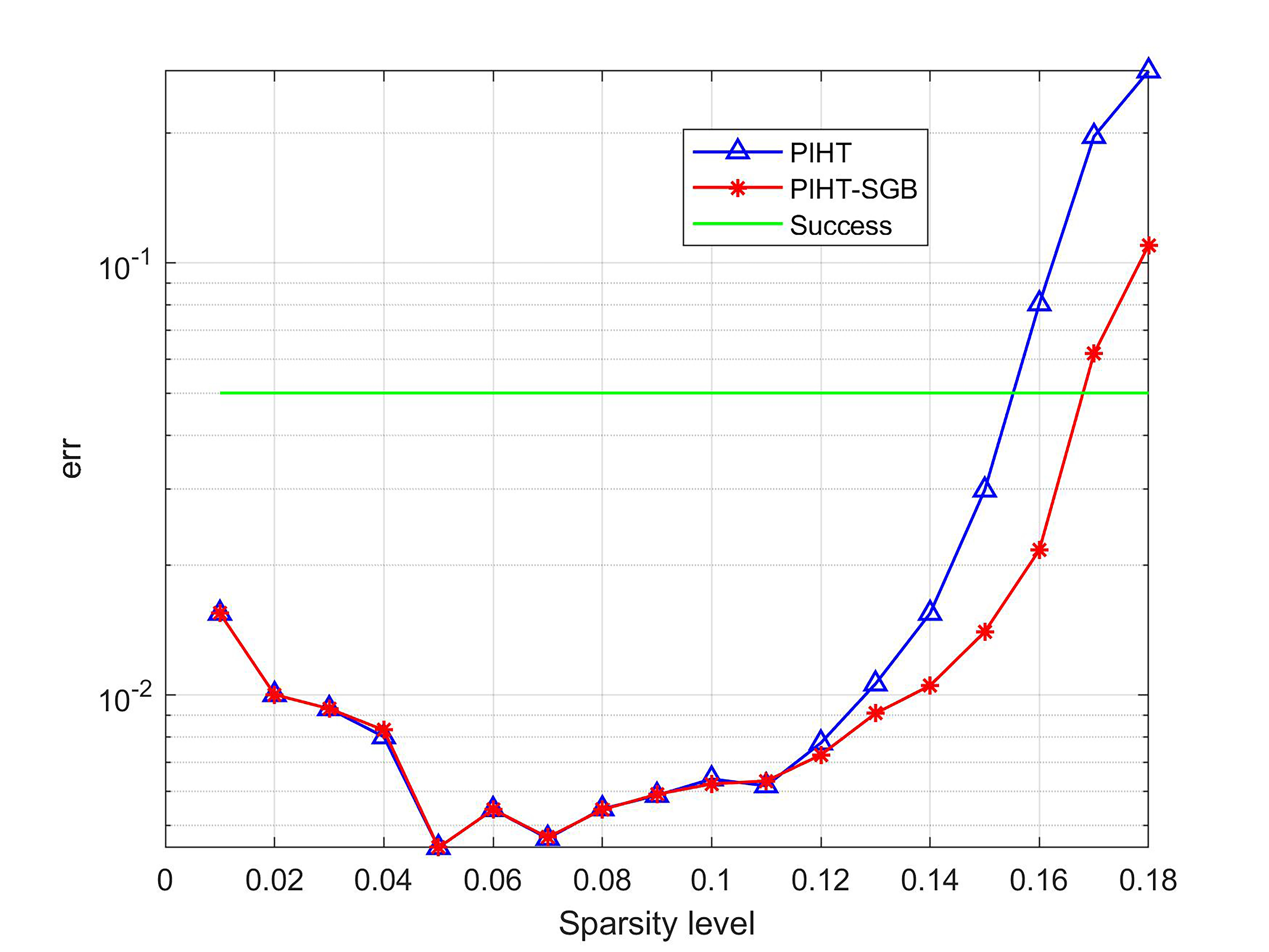}
	\caption{Comparison with PIHT under different sparsity levels.}
	\label{fig5-3}
\end{figure}
{As illustrated in Figure \ref{fig5-3}, with increasing $s$, PIHT-SGB achieve superior reconstruction accuracy compared to PIHT. These results demonstrate that jointly exploiting element-wise sparsity and group sparsity yields enhanced reconstruction performance over approaches that solely rely on element-wise sparsity.}
\subsection{Image reconstruction}
In this part, we will test PIHT-SGB in real application. We consider the multichannel image reconstruction problem \cite{7747459}. We compare PIHT-SGB with DCGL \cite{doi:10.1137/21M1443455}.

\textbf{E2} We consider the exact recovery $b = Ax^* + \sigma \xi$. $A \in \mathbb{R}^{m\times n}$ is random Gaussian matrix and the columns of $A$ are normalized to have $\ell_2$ norm of 1. In order to facilitate comparison, we adopt the same images ($48\times 48$) and experimental setting of \cite{7747459}. Each origin image has RGB (red, green, blue) three channels. Each image is transformed into an $n$-dimensional vector $x^*$ and the pixels are grouped at the same position from three channels together. Hence, the dimension $n = 48\times 48 \times 3 = 6912$ and each group $x_{G_i} \in \mathbb{R}^3$ for any $i\in [2304]$. $A\in \mathbb{R}^{1152\times 6912}$ is an random Gaussian matrix and the columns of $A$ are normalized to have $\ell_2$ norm of 1. The noise $\xi$ is coded as $\xi = {\rm{randn}}(n,\, 1)$ and $\sigma$ is set to $(0.08,0.09,0.1,0.12,0.13,0.14)$.
{According to \cite{6307860}, imposing constraints related to x facilitates image reconstruction by leveraging prior information. So, in this case, we set the boundary vector $l=u = 10\times {\rm{ones}}(n,\, 1)$, which is the same as it in \cite{doi:10.1137/21M1443455}.}
For image tasks, PSNR (peak signal to noise ration) is commonly used to quantify reconstruction quality, which is defined by
\begin{align*}
{\rm{PSNR}} := -10\log_{10}\left( \frac{\|x^k - x^*\|_{2}^{2}}{n} \right),
\end{align*}
it is obvious that the larger the PSNR value is, the better the image is reconstructed. For \textbf{E2}, we set $\epsilon = A^{-1}b$ in \eqref{eq-stop}.

\begin{figure}[htbp]
	\centering
	\includegraphics[width=10cm]{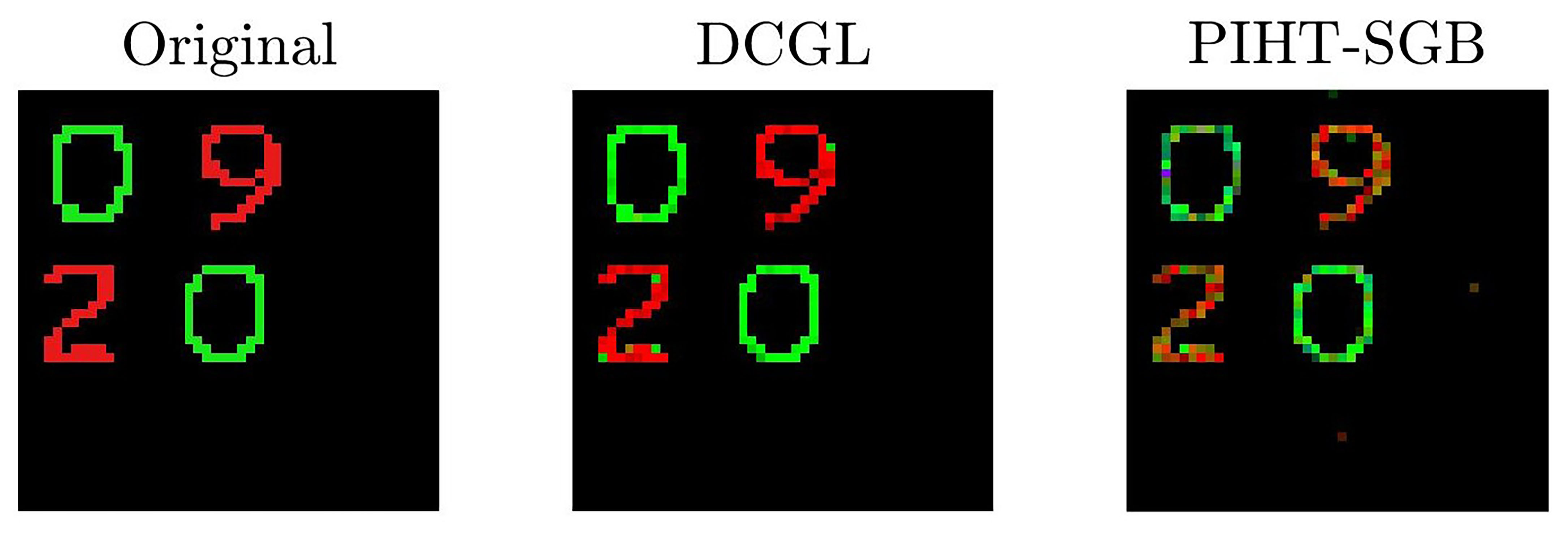
}
	\caption{Image recovery by DCGL and PIHT-SGB with $\sigma=0.08$.}
	\label{fig7}
\end{figure}
\begin{figure}[htbp]
	\centering
	\includegraphics[width=10cm]{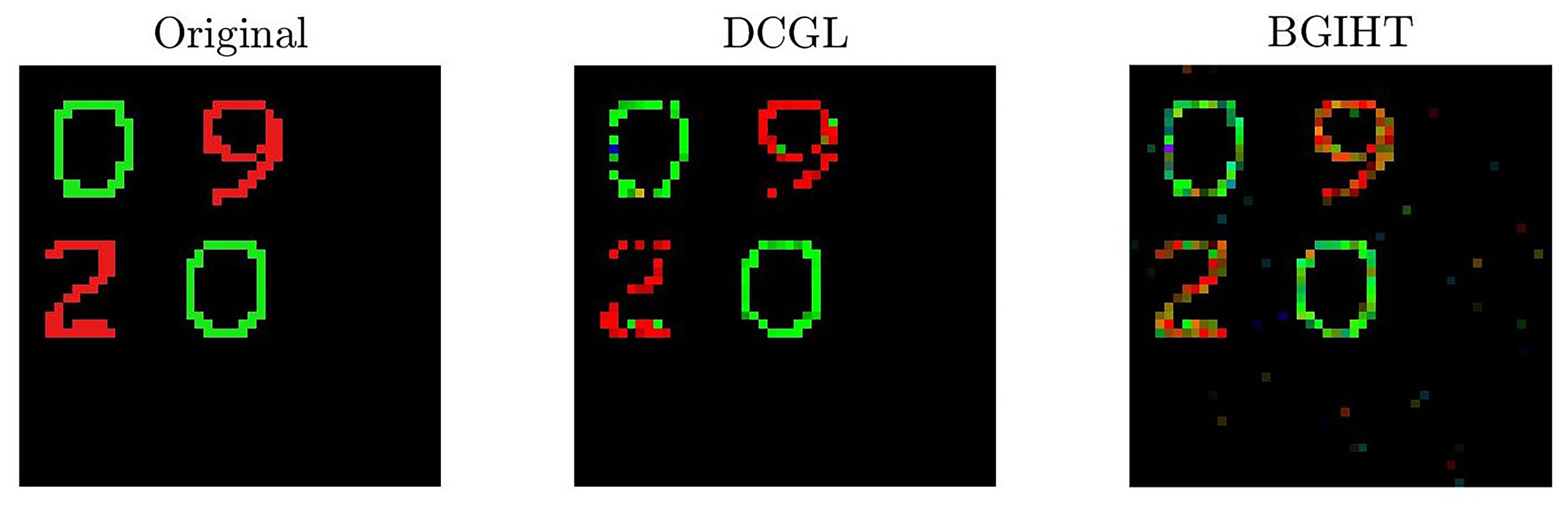
}
	\caption{Image recovery by DCGL and PIHT-SGB with $\sigma=0.12$.}
	\label{fig6}
\end{figure}

\begin{table*}[htbp]
\scriptsize
\centering
\caption{\text{Results on 2D-image recovery}.}\label{Tab03}
\begin{tabular*}{\textwidth}{@{\extracolsep\fill}ccccccc}
\toprule
\multirow{2}{*}{\text{Method}}
&\multicolumn{2}{c}{$\sigma = 0.08$}
&\multicolumn{2}{c}{$\sigma = 0.09$} &\multicolumn{2}{c}{$\sigma = 0.1$} \\
\cmidrule{2-3} \cmidrule{4-5} \cmidrule{6-7} &PSNR    &time (s)  &PSNR    &time (s)  &PSNR     &time (s)    \\
 \midrule
\text{DCGL} &27.66 &1.29 & 26.53  & 1.17 & 24.87  & 1.36  \\
\text{PIHT-SGB} &23.92 &0.4 & 23.97  & 0.38  & 23.68 & 0.38   \\
\midrule
\multirow{2}{*}{\text{Method}}
&\multicolumn{2}{c}{$\sigma = 0.12$} &\multicolumn{2}{c}{$\sigma = 0.13$} &\multicolumn{2}{c}{$\sigma = 0.14$} \\
\cmidrule{2-3} \cmidrule{4-5} \cmidrule{6-7} &PSNR    &time (s)  &PSNR     &time (s)    &PSNR    &time (s) \\
 \midrule
\text{DCGL} & 22.1  & 1.25 & 21.53  & 1.39  & 20.99  & 1.66   \\
\text{PIHT-SGB} & 22.51  & 0.42  & 21.88 & 0.43 & 21  & 0.79  \\
\bottomrule
\end{tabular*}
\end{table*}

As shown in Table \ref{Tab03}, Figure \ref{fig7} and Figure \ref{fig6}, we can see that, DCGL gives better results when $\sigma$ is smaller $(\sigma \le 0.1)$. We can also see that, our proposed PIHT-SGB returns higher PSNR while maintaining a relatively  efficient computational speed when $\sigma \in \{0.12,\, 0.13,\, 0.14\}$. It shows that PIHT-SGB has the advantage of dealing with larger noise cases.
\section{Conclusion}
In this paper, we proposed {a} proximal iterative hard thresholding method for addressing $ \ell_0 $ and $ \ell_{2,0} $ regularized optimization problem with box constraints. We derived a closed-form solution of the proximal operator for the considered optimization problem. Based on this proximal operator, we developed the proximal iterative hard thresholding method for sparse group $\ell_{0}$ regularized optimization with Box constraints (PHIT-SGB). We introduced the concepts of $\tau$-stationary point and support optimal (SO) point for \eqref{PBL0} and we established the relationship among them and the minimizer of \eqref{PBL0}. The global convergence of our proposed algorithm was established under standard assumptions. Finally, the numerical results highlighted the benefit of combining group sparsity terms with element-wise sparsity terms and demonstrated the efficiency of the proposed method.

\section{\textbf{Acknowledgement}}
We would like to thank Professor Wei Bian for sharing their code for DCGL, which is used in our paper as comparison. We would also like to thank two anonymous reviewers for wonderful comments, based on which, the paper is significantly improved.

\bibliography{sn-bibliography}

\end{document}